%% file: global-properties-part2.tex
\documentclass[11pt]{amsart}
\usepackage[colorlinks=true,allcolors=blue]{hyperref}
\usepackage{amssymb,latexsym,amsmath,tikz-cd,tikz}
\usepackage{enumitem}
\usepackage[mathscr]{eucal}
\usepackage{bbm,xcolor,comment}
\usepackage{mathdots}
\usepackage{chngcntr}

\numberwithin{equation}{section}
\numberwithin{table}{section} 
\numberwithin{figure}{subsection}

\input{lists.tex}
\input{macros.tex}

\input{thms.tex}

\theoremstyle{plain}

\newtheorem*{mainthmD}{Theorem${}^\dagger$ \ref{T:imageThetaI(A)}}
\newtheorem*{mainthmE}{Theorem${}^\dagger$ \ref{T:mark}}
\newtheorem*{mainthmF}{Corollary${}^\dagger$ \ref{C:mark}}
\newtheorem*{mainthmG}{Corollary${}^\dagger$ \ref{C:thatsall}}

\def\tsum{\textstyle{\sum}}

\def\olB{{\overline{B}}}

\def\sD{\mathscr{D}}
\def\ff{\mathfrak{f}}
\def\cG{\mathcal{G}}
\def\sH{\mathscr{H}}

\def\cI{\mathcal{I}}

\def\tPhi{\widetilde\Phi}

\def\sM{\mathscr{M}}

\def\sfn{\mathsf{n}}
\def\sS{\mathscr{S}}
\def\sU{\mathscr{U}}
\def\olU{\overline\sU}

\def\Zc{Z^\mathrm{c}}
\def\Zw{Z^\mathrm{w}}

\def\CKS{MR840721}
\def\CMSP{MR3727160}

\begin{document}
\title{Period maps at infinity}

\author[Green]{Mark Green}
\email{mlg@math.ucla.edu}
\address{UCLA Mathematics Department, Box 951555, Los Angeles, CA 90095-1555}

\author[Griffiths]{Phillip Griffiths}
\email{pg@math.ias.edu}
\address{Institute for Advanced Study, 1 Einstein Drive, Princeton, NJ 08540}
\address{University of Miami, Department of Mathematics, 1365 Memorial Drive, Ungar 515, Coral Gables, FL  33146}

\author[Robles]{Colleen Robles}
\email{colleen.robles@duke.edu}
\address{Mathematics Department, Duke University, Box 90320, Durham, NC  27708-0320} 
\thanks{Robles is partially supported by NSF DMS 1611939, 1906352.}

\date{\today}

\begin{abstract}
Let $\olB$ be a smooth projective varieity, and $Z \subset \olB$ a simple normal crossing divisor.  Assume that $B = \olB - Z$ admits a variation of pure, polarized Hodge structure.  The divisor $Z$ is naturally stratified, and Schmid's nilpotent orbit theorem defines a family/variation of nilpotent orbits along each strata.  We study the rich geometric structure encoded by this family, its relationship to the induced (quotient) variation of pure Hodge structure on the strata, and establish a relationship between the extension data in the nilpotent orbits and the normal bundles of the smooth irreducible components of $Z$.
\end{abstract}

\maketitle

\setcounter{tocdepth}{1}
\let\oldtocsection=\tocsection
\let\oldtocsubsection=\tocsubsection
\let\oldtocsubsubsection=\tocsubsubsection
\renewcommand{\tocsection}[2]{\hspace{0em}\oldtocsection{#1}{#2}}
\renewcommand{\tocsubsection}[2]{\hspace{3em}\oldtocsubsection{#1}{#2}}

\section{Introduction} 

\subsection{The setup} \label{S:setup}

Fix a smooth projective variety $\olB$ with simple normal crossing divisor $Z$.  Suppose the complement $B = \olB - Z$ admits a variation of (pure) polarized Hodge structure with local system
\[
\begin{tikzcd}[row sep=small,column sep=tiny]
  \mathbb{V} \arrow[r,equal] \arrow[d] &
  \tilde B \times_{\pi_1(B)} V_\bZ 
  \\ B &
\end{tikzcd}\]
having unipotent local monodromy around $Z$, and Hodge bundles
\begin{equation} \nonumber 
  \cF^p \ \subset \ 
  \cV \ = \ \tilde B \times_{\pi_1(B)} V_\bC \,.
\end{equation}
Here $V_\bZ$ is a lattice, $\tilde B \to B$ is the universal cover,  $\rho : \pi_1(B) \to \tGL(V_\bZ)$ is the monodromy representation, and $V_\bC = V_\bZ \ot_\bZ \bC$.  Let
\begin{equation} \nonumber 
  \Phi : B \ \to \ \Gamma \bs \sD 
\end{equation}
be the induced period map.  Here $\sD$ is a period domain parameterizing pure, weight $\sfn$, $Q$--polarized integral Hodge structures on $V_\bZ$; and $\Gamma = \rho(\pi_1(B))$ is the image of the monodromy representation.  Applying a Tate-twist if necessary, we may assume that the Hodge structures parameterized by $\sD$ are effective.  

\subsection{Nilpotent orbits at infinity} \label{S:nilp-orb-infty}

Write 
\[
  Z \ = \ Z_1 \,\cup\, Z_2 \,\cup\cdots\cup\, Z_\nu \,,
\]
with smooth irreducible components $Z_i$.  We denote by 
\[
  Z_I \ = \ \cap_{i\in I} \, Z_i
\]
the closed strata, and
\[
  Z_I^* \ = \ Z_I \,-\, \cup_{j\not\in I} \, Z_j 
  \ = \ \cap_{i \in I} \, Z_i^*
\] 
the smooth strata.  As we approach a point $o \in Z_I^*$ (a local lift of) the period map $\Phi$ degenerates to a limiting mixed Hodge structure $(W,F_o)$ that is polarized by a cone $\s_I \subset \fgl(V_\bQ)$ of nilpotent operators (arising from logarithms of the local monodromy around $o$).  The Hodge filtration $F_o$ depends on a choice of local coordinates at $o$, and is defined only up to the action of $\exp(\bC\s_I) \subset \tGL(V_\bC)$ on the compact dual $\check \sD$; here
\[
  \bC\s_I \ = \ \tspan_\bC\{\s_I\} \ \subset \ \fgl(V_\bC)\,.
\]
The orbit $\exp(\bC\s_I) \cdot F_o \subset \check \sD$ is independent of our choice of local coordinates.  The triple $(W,\,\s_I,\,\exp(\bC\s_I)\cdot F_o)$ depends on our choice of local lift, and so is well-defined only up to the action of $\Gamma$.  The $\Gamma$--conjugacy classes of $W$ and $\s_I$ are locally constant on $Z_I^*$.  The $\Gamma$--conjugacy class of the orbit $\exp(\bC\s_I) \cdot F_o \subset \check \sD$ may vary along $Z_I^*$.

For notational convenience we assume that the strata $Z_I$ are connected.\footnote{This may always be arranged after replacing $\olB$ with a suitable log modification $\hat B \to \olB$.  Alternatively, one may index the connected components $Z^*_{I,s} \subset Z_I^*$, and modify the arguments that follows accordingly.}  Then $Z_I^*$ is connected, and the $\Gamma$--conjugacy class of the pair $(W,\s_I)$ is constant along $Z_I^*$.  Fix an element of this conjugacy class.  Let $\sM_I \subset \check \sD$ be the Hodge filtrations $F$ in the compact dual with the property that $(W,F)$ is a mixed Hodge structure polarized by the nilpotent operators $N \in \s_I$.  Let $\Gamma_I \subset \Gamma$ be the subgroup centralizing the cone $\s_I$.  This group also stabilizes the weight filtration $W$.   Then we obtain map 
\begin{equation}\label{E:PsiI}
  \Psi_I : Z_I^* \ \to \ 
  (\Gamma_I\exp(\bC\s_I))\backslash \sM_I \,, 
\end{equation}
cf.~\S\ref{S:PsiI}.

\subsection{Goal and motivation}

The goal of this paper is to study the structure of the maps $\Psi_I$.  This is motivated by the idea that we should be able to use the maps $\Psi_I$ to construct a Hodge--theoretically meaningful algebraic completion of the period map $\Phi$.  In the case that $\sD$ is hermitian and $\Gamma$ is neat the toroidal compactifications $\Gamma \bs \sD_\Sigma$ of Ash--Mumford--Rapoport--Tai \cite{MR0485875, MR0457437} may be seen, a posteori, to be of this form (where we take the period map to be the identity $\Gamma \bs \sD \to \Gamma \bs \sD$).  In the general case ($\sD$ not necessarily hermitian) Kato--Usui \cite{MR2465224}, with refinements by Kato--Nakayama--Usui \cite{MR2721860, MR3084721}, have proposed a generalization of AMRT's construction that, if realized, would also yield an extension $\Phi_\Sigma : \olB \to \Gamma \bs \sD_\Sigma$.  The problem here is to find a ``$\Gamma$--strongly compatible weak fan'' $\Sigma$, consisting of nilpotent cones $\tau\subset \fg_\bQ$, and having the property that for every $(\s_I,F_o)$ arising as in \S\ref{S:nilp-orb-infty} there exists a unique minimal $\tau \in \Sigma$ so that $\s_I \subset \tau$ and $(\tau,F_o)$ also defines a nilpotent orbit.  Then $\Phi_\Sigma$ maps $o \in Z_I^*$ to the $\Gamma$--conjugacy class of the nilpotent orbit $\exp(\bC\tau)\cdot F_o$.  In particular, the restriction $\left.\Phi_\Sigma\right|_{Z_I^*}$ is of the form \eqref{E:PsiI}, but with $\tau$ in place of $\s_I$.

\begin{remark}[Related work by Chen, Deng and Robles]
The construction of the fan $\Sigma$ is trivial when $\tdim\,B=1$.  The first nontrivial\footnote{By ``nontrivial'' we mean that $\tdim\,B \ge 2$ and the image of the period map does not factor through a locally hermitian symmetric space.} example of a KNU completion $\Phi_\Sigma : \olB \to \Gamma \bs \sD_\Sigma$ with $\tdim\,B=2$ was given by H.~Deng \cite{MR4441155}.  This was shortly followed by a second example by C.~Chen \cite{Chen1221}.  In both cases, $B$ parameterizes families of Calabi--Yau varieties arising as mirrors to complete intersections in toric varieties.  Recently Deng and Robles have shown that every period map with $\tdim\,B=2$ admits a KNU completion \cite{DRdim2}.  It remains an open, and seemingly difficult, problem to demonstrate the existence of compatible weak fan $\Sigma$ when $\tdim\,B\ge3$.  Nonetheless the $\Psi_I$ can be patched together to define an algebraic completion of the Stein factorization of the period map \cite{deng-gentor, DRstein}.  From a Hodge theoretic perspective these completions all encode the same same information at infinity: conjugacy classes of nilpotent orbits.
\end{remark}

\subsection{The structure of $\Psi_I$}

Returning to our present goal, the Hodge filtration $F_o$ induces a pure Hodge structure on the quotient spaces $\tGr^W_\ell = W_\ell/W_{\ell-1}$, and $N \in \s_I$ determines a polarized sub-Hodge structure $P_\ell(N) \subset \tGr^W_\ell$, cf.~\S\ref{S:LMHS-F(w)}.  In this way, we obtain a period map $\Phi_I : Z_I^* \to \Gamma_I \bs \sD_I$ factoring through $\Psi_I$, cf.~\S\ref{S:PhiI}.  By considering the structures that $\Psi_I$ induces on $W_\ell/W_{\ell-2}$ we obtain an intermediate map $\Theta_I$, and commutative diagram
\begin{equation}\label{E:PhiTheta}
  \begin{tikzcd}
    Z_I^* \arrow[r,"\Psi_I"]  
    	\arrow[rd,"\Theta_I"',end anchor={west}]
    	\arrow[rdd,"\Phi_I"',end anchor={west},bend right]
    & (\Gamma_I \exp(\bC\s_I))\bs \sM_I 
		\arrow[d,"\pi_1"] \\
    & \Gamma_I \bs \sM_I^1 \arrow[d,"\pi_0"] \\
    & \Gamma_I \bs \sD_I \,,
  \end{tikzcd}
\end{equation}
cf.~\S\ref{S:ThetaI}.  This allows us to study the structure of $\Psi_I$ in two steps: (i) the relationship between $\Phi_I$ and $\Theta_I$; and (ii) the relationship between $\Theta_I$ and $\Psi_I$.

There is a rich geometric structure relating the maps $\Phi_I$ and $\Theta_I$:

\begin{mainthmD}
The fibres of $\pi_0 : \Gamma_I \bs \sM_I^1 \to \Gamma_I \bs \sD_I$ are finite quotients of complex tori.  Each torus contains an abelian variety.  Let $A$ be a connected component of a $\Phi_I$--fibre.  \emph{(In particular, $\Theta_I(A)$ is contained in a $\pi_0$--fibre).}  Then $\Theta_I(A)$ is contained in a finite quotient of a translate of the abelian variety.  The finite quotients are trivial if $\Gamma$ is neat.
\end{mainthmD}

\noindent (In the interest of conciseness, some results discussed in this Introduction are stated imprecisely and/or incompletely; this is indicated by the superscript ${}^\dagger$.  The reader will find the complete and precise statements, with all necessary definitions, in the body of the paper.)

\begin{remark}[Related work of Kerr and Pearlstein]
The map $\pi_0$ of \eqref{E:PhiTheta} is one piece of a fibration tower interpolating between $(\Gamma_I \exp(\bC\s_I)) \bs \sM_I$ and $\Gamma_I \bs \sD_I$.  This tower is studied in \cite[\S7]{MR3474815}, where it is shown that the iterated fibres are generalized intermediate Jacobians.
\end{remark}

\begin{remark}[Related work of Bakker, Brunebarbe, Klingler and Tsimerman]
The triple $(W,F_o,\s_I)$ of \S\ref{S:nilp-orb-infty} defines a graded-polarized mixed Hodge structure.  If the $\Gamma$--conjugacy class of the Hodge filteration $F_o \in \check \sD$ were well-defined, then we would obtain a variation of graded-polarized mixed Hodge structures along $Z_I^*$.  However, in our situation it is only the $\Gamma$--conjugacy class of the orbit $\exp(\bC\s_I)\cdot F_o \subset \check \sD$ that is well-defined; the map $\Psi_I$ of \eqref{E:PsiI} may be viewed as the quotient of a variation of graded-polarized mixed Hodge structures.  The (mixed) period map of an admissible graded-polarized integral mixed Hodge structure is known to be $\mathbb{R}_\mathrm{an,exp}$--definable \cite{MR4742809}.  We anticipate that the ideas there may be adapted to show that the map $\Psi_I : Z_I^* \to (\Gamma_I \exp(\bC\s_I)) \bs \sM_I$ of \eqref{E:PsiI} is $\mathbb{R}_\mathrm{an,exp}$--definable, but do not pursue this here.
 
The definibility result of \cite{MR4742809} was used by Bakker--Brunebarbe--Tsimerman to show that the image of a (proper, mixed) period map is quasi-projective \cite{MR4661535}.  In that work, they established an analog of Theorem \ref{T:imageThetaI(A)} for mixed period maps \cite[Proposition 2.17]{MR4661535}.  They also show that the theta bundle is relatively ample over the image of the (pure) period map induced by taking the weight-graded quotient of the mixed period map \cite[Corollary 2.18]{MR4661535}.  This result is analogous to Theorem \ref{T:thatsall} (discussed below), and the analogy is strongest under certain nondegeneracy conditions on the cones, cf.~Corollary \ref{C:thatsall}.
\end{remark}

The next result relates the geometry \emph{along} the $\Phi_I$--fibres to the geometry \emph{normal} to $Z$. 

\begin{mainthmE}
There exist line bundles $\Lambda_M$ over $\Gamma_I \bs \sM_I^1$ polarizing the abelian varieties of Theorem \ref{T:imageThetaI(A)} so that
\begin{equation}\label{E:mark-intro}
  \left.\Theta_I^*(\Lambda_M)\right|_A
  \ = \ 
  \sum_j \left.\sQ(M,N_j) [Z_j]\right|_A \,,
\end{equation}
with $\sQ(M,N_j) \in \bZ$, and summing over all $Z_j \cap A \not=\emptyset$.
\end{mainthmE}

\begin{mainthmF}
If the differential of $\left.\Theta_I\right|_A$ is injective, then the line bundle $-\sum\left.\sQ(M,N_j)\,\cN^*_{Z_j/\olB}\right|_A$ is ample. 
\end{mainthmF}

Surprisingly, the map $\Psi_I$ is almost completely determined by $\Theta_I$ (Theorem \ref{T:thatsall}).  Moreover, the information in $\Psi_I$ that is not captured by $\Theta_I$ is encoded by certain line bundles and their sections (Remark \ref{R:thatsall}).  We defer the precise statements of these results to \S\ref{S:Psi-v-Theta}, as they involve a somewhat subtle relationship between proper extensions of $\Psi_I$ and $\Theta_I$.  A special case of Theorem \ref{T:thatsall} is

\begin{mainthmG}
Given $j \not \in I$, assume that the weight filtrations $W(\s_I) \not= W(\s_{I \cup\{j\}})$ do not coincide. Then $\Psi_I$ is locally constant on the fibres of $\Theta_I$. 
\end{mainthmG} 

\begin{remark}
The implication of Theorem \ref{T:thatsall} is that \eqref{E:mark-intro} \emph{is the central geometric information that arises when considering the variation $\Psi_I$ along the $\Phi_I$--fiber}.
\end{remark}

\subsection*{Acknowledgements}

This paper grew out of discussions and work \cite{GGLR} with Radu Laza.  CR is indebted to Haohua Deng for many illuminating discussions, including those that led to Lemma \ref{L:nbd-ZI}.

\tableofcontents 

\section{Review of local behavior at infinity} 

Here we set notation and review well-known properties of period maps and their local behavior at infinity.  Good references for this material include \cite{\CMSP, \CKS, MR2918237, MR0259958, MR3474815, MR0382272}.

\subsection{Group notation}

Given a ring $\bZ \subset R \subset \bC$, define $V_R = V_\bZ \ot_\bZ R$.  The polarization is a nondegenerate bilinear form $Q : V_\bQ \times V_\bQ \to \bQ$ satisfying
\[
  Q(u,v) \ = \ (-1)^\sfn Q(v,u) \,,
  \quad\hbox{for all}\quad u,v \in V_\bQ \,.
\]
Let $\tGL(V_R) \simeq \tGL_r(R)$ be the group of invertible $R$--linear maps $V_R \to V_R$.  And let 
\[
  G_R \ = \ \tGL(V_R,Q)
  \ = \ 
  \{ g \in \tGL(V_R) \ | \ Q(gu,gv) = Q(u,v) \,,\ 
  \forall \ u,v \in V_R\}
\]
be the subgroup preserving the polarization.  We have
\[
  \Gamma \ \subset \ G_\bZ \,.
\]
Let $\fgl(V_R) \simeq \fgl_r(R)$ be the Lie algebra of $R$--linear maps $V_R \to V_R$.  Set
\[
  \fg_R \ = \ \fgl(V_R,Q) \ = \ 
  \{ \xi \in \fgl(V_R) \ | \ 
  0 = Q(\xi u, v) +Q(u,\xi v) \,,\ \forall \ u,v \in V_R \} \,.
\]
When $R = \bR,\bC$, $G_R$ is a Lie group with Lie algebra $\fg_R$.

\subsection{Period maps at infinity} \label{S:vhs}

\subsubsection{}

Let 
\[
  \Delta \ = \ \{ t \in \bC \ | \ |t|<1\}
\]
denote the unit disc, and 
\[
  \Delta^* \ = \ \{ t \in \bC \ | \ 0<|t|<1\} 
\]
the punctured unit disc.  The upper half plane
\[
  \sH \ = \ \{ z \in \bC \ | \ \tIm\,z>0 \}
\]
is the universal cover of $\Delta^*$, with covering map
\[
  \sH \ \to \ \Delta^*
  \quad\hbox{sending}\quad
  z \ \mapsto \ t = e^{2\pi\bi z} \,.
\]
Let 
\[
  \ell(t) \ = \ \frac{\log t}{2\pi\bi}
\]
denote the multi-valued inverse.

\subsubsection{} 

Set $|I| = \tcodim\,Z_I$.  Then $I = \{i_1,\ldots,i_k\}$, with $k = |I|$.  Fix a point $o \in Z_I^* \subset \olB$.    We may choose a local coordinate chart 
\[
  (t,w) : \olU \,\subset\,\olB \ \stackrel{\simeq}{\longrightarrow} \ 
  \Delta^{k+r}
\]
centered at $o$, so that 
\[
  \olU \,\cap\, Z_{i_a} \ = \ 
  \olU \cap Z_{i_a}^* \ = \ 
  \{ t_a=0\} \,,
  \quad \hbox{for all}\quad 1 \le a \le k \,,
\]
and
\[
  (t,w) : \sU \,=\,B \,\cap\,\olU \ \stackrel{\simeq}{\longrightarrow} \ 
  (\Delta^*)^k \times \Delta^r \,.
\]

\subsubsection{} \label{S:Ti}

Given any point $(t,w) \in \sU$ and $i_a \in I$, the closed curve parameterized by 
\[
  \mathbf{c}_{i_a}(s) \ = \ (t_1,\ldots,t_{a-1},\,t_a\,e^{2\pi\bi s},\, t_{a+1},\ldots,t_k;\,w) \,,
  \quad 0 \le s \le 1\,,
\]
is contained in $\sU$ and circles $Z_{i_a}$.  These curves define counter-clockwise generators $\{[\mathbf{c}_{i_a}]\}_{i=1}^k$ of the fundamental group $\pi_1(\sU , (t,w)) \simeq \pi_1((\Delta^*)^k \times \Delta^r) \simeq \bZ^k$.  Parallel transportation (by the Gauss--Manin connection) along the curve $\mathbf{c}_{i_a}$ defines an operator $\mathbf{T}_{i_a}(t,w) \in \tGL(\mathbb{V}_{(t,w)})$ that depends only on the homotopy class $[\mathbf{c}_{i_a}] \in \pi_1(\sU;(t,w))$.  These operators are the \emph{local monodromy about $Z$}.  In general they are quasi-unipotent.  In this paper we are assuming that the operators are unipotent (\S\ref{S:setup}).  Each $\mathbf{T}_{i_a}$ is a flat section of $\tGL(\mathbb{V})$ over $\sU$, \cite{MR0382272}.  

\subsubsection{} \label{S:loclift}

Each flat section $\mathbf{T}_{i_a} \in H^0(\sU,\left.\tGL(\mathbb{V})\right|_\sU)$ of \S\ref{S:Ti} determines a $\Gamma$--conjugacy class of unipotent operators $T_{i_a} \in \tGL(V_\bZ)$.  More generally, the $k$-tuple of flat sections $\{\mathbf{T}_{i_a}\}_{a=1}^k \subset H^0(\sU,\left.\tGL(\mathbb{V})\right|_\sU)$ in \S\ref{S:Ti} determines a $\Gamma$--conjugacy class $\cT_\sU \subset \tGL(V_\bZ) \,\times\cdots\times\, \tGL(V_\bZ)$.  A choice of element $\{T_{i_a}\}_{i_a \in I} \in \cT_{\sU}$ in the $\Gamma$--conjugacy class determines a local lift $\Phi_\sU$ of the period map as follows.  Let $\Gamma_\sU \subset \Gamma$ be the subgroup generated by $\{T_{i_a}\}_{a=1}^k$.  There is a commutative diagram
\begin{equation}\label{E:PhisU} \begin{tikzcd}  
  & \Gamma_\sU \bs \sD \arrow[d] \\
  \sU \arrow[r,"\Phi"'] \arrow[ru,"\Phi_\sU"]
  & \Gamma \bs \sD \,.
\end{tikzcd}\end{equation}
The nilpotent orbit theorem \cite{MR0382272} describes the structure of the local lift $\Phi_\sU$, as follows.

By hypothesis (\S\ref{S:Ti}) the $\mathbf{T}_{i_a}(t,w): \mathbb{V}_{(t,w)} \to \mathbb{V}_{(t,w)}$ are unipotent.  Equivalently, the $T_{i_a} : V_\bZ \to V_\bZ$ are unipotent.  Let 
\[
  N_{i_a} \ = \ \log T_{i_a} \ \in \ \fgl(V_\bQ)
\]
be the logarithm of $T_{i_a}$.  The universal cover of $\sU$ is $\widetilde \sU = \sH^k \times \Delta^r$, and we have a commutative diagram
\[ \begin{tikzcd} 
  \widetilde\sU \arrow[r,"\tPhi"] \arrow[d]
  & \sD \arrow[d] \\
  \sU \arrow[r,"\Phi_\sU"]
  & \Gamma_\sU \bs \sD \,.
\end{tikzcd} \]
The lift to the universal cover is of the form
\begin{equation}\label{E:tPhi}
  \tPhi(t,w) \ = \ \exp( \tsum \ell(t_a) N_{i_a} ) 
  \,g(t,w) \cdot F_o \,.
\end{equation}
Here, $F_o$ is an element of the compact dual $\check \sD \supset \sD$ (which is the flag variety parameterizing the filtrations $F^p(V_\bC)$ that satisfy the first Hodge--Riemann bilinear relation
$Q(F^p,F^q) = 0$ for all $p+q > \sfn$,
but not necessarily the second); the group $G_\bC$ acts transitively on $\check \sD$, and 
\begin{equation}\label{E:tilde-g}
  g : \olU \ \to \ G_\bC
\end{equation}
is a holomorphic map; and we abuse notation by conflating the multi-valued $\ell(t_a)$ with the coordinates $z_a$ on $\sH^k$.

\subsection{Limiting mixed Hodge structures} \label{S:LMHS}

Fix an element $\{T_{i_a}\}_{i_a \in I}$ of the $\Gamma$--conjugacy class $\cT_\sU$.  Let 
\begin{equation} \nonumber 
  \s_I \ = \ \{ y_1 N_{i_1} + \cdots + y_k N_{i_k} \ | \ 0 < y_a \in \bQ \}
  \ \subset \ V_\bQ
\end{equation}
be the rational cone generated by the nilpotent $N_{i_a} = \log T_{i_a}$.

\subsubsection{}  \label{S:W}

A nilpotent operator $N\in\s_I$ determines a rational, increasing filtration $W_0 \subset W_1 \subset \cdots \subset W_{2\sfn} = V_\bQ$, \cite{MR1251060}.  This is the unique filtration satisfying the conditions \eqref{SE:W(N)}: first, 
\begin{subequations}\label{SE:W(N)}
\begin{equation} \label{E:W(N)}
  N(W_\ell) \ \subset \ W_{\ell-2} \,.
\end{equation}
If we set 
\[
  \tGr^W_\ell \ = \ W_\ell/W_{\ell-1} \,,
\]
then \eqref{E:W(N)} implies that $N$ induces a well-defined map $N: \tGr^W_\ell \to \tGr^W_{\ell-2}$.  The map 
\begin{equation}\label{E:Niso}
  N^k : \tGr^W_{\sfn+k} \ \to \ \tGr^W_{\sfn-k}
  \quad\hbox{is an isomorphism for all } k \ge 0 \,.
\end{equation}
\end{subequations}

\subsubsection{} \label{S:LMHS-F(w)}

Any two $N , N' \in \s_I$ determine the same filtration $W$, and we call $W = W(\s_I)$ the \emph{weight filtration} of the monodromy cone, \cite[Theorem 3.3]{MR664326}.  If $F(w) = g(0,w)\cdot F_o$, then $(W,F(w))$, is a mixed Hodge structure (MHS) that is polarized by $\s_I$.  This means that $F(w)$ induces a weight $\ell$ Hodge structure $\tGr^W_\ell$; for each $N \in \s_I$, the kernel
\[
  P_{\sfn+k}(N) \ = \ 
  \tker\{ N^{k+1} : \tGr^W_{\sfn+k} \to \tGr^W_{\sfn-k-2}\}
\]  
is a rational Hodge substructure; and the weight $\sfn+k$ Hodge structure on $P_{\sfn+k}(N)$ is polarized by $Q(\cdot , N^k \cdot )$.
The triple $(W,F(w),\s_I)$ is a \emph{limiting} (or \emph{polarized}) \emph{mixed Hodge structure} (LMHS); and we say $\s_I$ \emph{polarizes} $(W,F(w))$.  

\subsubsection{}  \label{S:MI}

The set
\[
  \sM_I \ = \ \{ F \in \check \sD \ | \ (W,F) 
  \hbox{ is a MHS polarized by } \s_I \} 
\]
is a complex submanifold.  A subgroup $\cG_I \subset G_\bC$ acts on $\sM_I$ by biholomorphism.  To describe the group $\cG_I$, let $P_W \subset G$ be the parabolic subgroup stabilizing $W$.  The unipotent radical 
\[
  P_W^{-1} \ = \ \{ g \in P_W \ | \ g \hbox{ acts as the identity on }
  \tGr^W_\ell \,,\ \forall \ \ell \}
\]
is the subgroup of $P_W$ acting trivially on $\tGr^W_\ell$, for all $\ell$.  Since $W = W(\s_I)$, the centralizer
\[
  C_I \ = \ \{ g \in G \ | \ \tAd_g(N) = N \,,\ \forall \ N \in \s_I \}
\] 
of the cone $\s_I$ is a subgroup of $P_W$
\begin{equation}\label{E:CinP}
   C_I \ \subset \ P_W \,.
\end{equation}
The group 
\[
  C_I^{-1} \ = \ C_I \,\cap\, P_W^{-1}
\]
is the unipotent radical of $C_I$, and a normal subgroup. The group
\begin{equation}\label{E:cGI}
  \cG_I \ = \ C_{I,\bR} \cdot C_{I,\bC}^{-1} 
  \ \subset \ P_{W,\bC}
\end{equation}
acts on $\sM_I$.    

\begin{remark} \label{R:GonM}
The action of $\cG_I$ on $\sM_I$ is almost transitive: $\sM_I$ consists of finitely many connected components, and the connected identity component $\cG_I^\circ \subset \cG_I$ acts transitively on each connected component of $\sM_I$, \cite{MR3474815}.  In particular, each $\cG_I$--orbit is open and closed in $\sM_I$ (a union of connected components).  In a mild abuse of notation, we will conflate $\sM_I$ with the connected component $\sM_I^\circ \subset \sM_I$ containing the $F_o$, with $o \in Z_I^*$, and $\cG_I$ with the connected identity component $\cG_I^\circ \subset \cG_I$.
\end{remark}

\begin{remark}
Note that $\exp(\bC\s_I)$ is a normal subgroup of $\cG_I$, and $\exp(\bC\s_I) \bs \cG_I$ acts transitively on the quotient manifold $\exp(\bC\s_I)\bs \sM_I$.
\end{remark}

\begin{lemma}\label{L:good-quotients}
The quotients $\sM_I \to \Gamma_I\bs \sM_I$ and $\exp(\bC\s_I) \bs 
\sM_I \to (\Gamma_I \exp(\bC\s_I)) \bs \sM_I$ are morphisms of complex analytic spaces.    
\end{lemma}

\noindent The first half of the lemma is \cite[Corollary 3.8]{MR4742809}.  The second half seems to be ``known to the experts''; we give a proof in \S\ref{S:p} for completeness.

\subsection{Induced maps along strata $Z_I^*$}

We continue to work with the representative $\{T_{i_a}\}_{i_a \in I}$ of the $\Gamma$--conjugacy class $\cT_\sU$ that was fixed in \S\ref{S:loclift} and \S\ref{S:LMHS}. 

\subsubsection{} \label{S:PsiI}

The map 
\begin{equation}\label{E:FI}
  F_I \,:\, Z_I^* \cap \olU \ \to \ \sM_I \,,\quad
  w \mapsto F_I(w) = g(0,w) \cdot F
\end{equation}
defines a variation of limiting mixed Hodge structure $(W,F_I(w),\s_I)$ over $Z_I^* \cap \olU$.  The map \eqref{E:FI} is not well-defined; it depends on our choice of local coordinates $(t,w)$.  What is well-defined is the composition 
\[
\begin{tikzcd}
  Z_I^* \cap \olU 
  \arrow[r,"F_I"] & 
  \sM_I \arrow[r,"\nu_I"] &
  \exp(\bC\s_I)\bs \sM_I \,.
\end{tikzcd}
\]
(It is the nilpotent orbit that is well-defined.)  

The composition $\nu_I \circ F_I : Z_I^* \cap \olU \to \exp(\bC\s_I)\bs \sM_I$ is a holomorphic map.  In order to patch these together in to a well-defined map along all of $Z_I^*$ we need the following lemma.  Let
\[
  \Gamma_I \ = \ C_{I,\bQ} \,\cap\, \Gamma
\]
be the monodromy subgroup centralizing the cone.

\begin{lemma}[{\cite{DRstein}}] \label{L:nbd-ZI}
There exists a neighborhood $X \subset \olB$ of $Z_I$ so that the restriction of the period map $\Phi$ to $U = B \cap X$ lifts to $\Gamma_I \bs \sD$: there is a holomorphic map $\Phi_U : U \to \Gamma_I\bs \sD$ so that the diagram 
\begin{equation} \nonumber 
\begin{tikzcd}
  & \Gamma_I \bs \sD \arrow[d] \\
  U \arrow[r,"\Phi"'] \arrow[ru,"\Phi_U"]
  & \Gamma \bs \sD
\end{tikzcd}
\end{equation}
commutes.
\end{lemma}

Informally, we say \emph{the monodromy near $Z_I^*$ takes value in $\Gamma_I$}.

\begin{proof}
The lemma generalizes the construction of the local lift \eqref{E:PhisU}.  Note that $\Gamma_I$ is also the group centralizing the $T_{i_a} = \exp(N_{i_a}) \in \tGL(V_\bZ)$.  So, to prove the lemma, it suffices to show that the flat sections $\mathbf{T}_{i_a}(t,w) \in H^0(\sU,\tGL(\left.\mathbb{V})\right|_\sU)$ constructed in \S\ref{S:Ti} extend to a punctured neighborhood $U = B \cap X$ of $Z_I$.

Given any point $o \in Z_I$, there is a unique $J = \{j_1,\ldots,j_\ell\} \supset I$ so that $o \in Z_J^*$.  Fix a coordinate neighborhood $\olU_o \subset \olB$ centered at $o$, as in \S\ref{S:vhs}.  Given any point $(t,w) \in \sU_o = B \cap \olU_o$ and $i_a \in I \subset J$, the loop $\mathbf{c}_{i_a}(s) = (t_1,\ldots,t_{a-1},\,t_a\,e^{2\pi\bi s},\, t_{a+1},\ldots,t_\ell;\,w)$, $0 \le s \le 1$, is contained in $\sU_o$ and circles $Z_{i_a}$.  Parallel transportation (by the Gauss--Manin connection) around this loop defines the unipotent operator $\mathbf{T}_{i_a}(t,w) \in \tGL(\mathbb{V}_{(t,w)})$.  

Define $X = \bigcup_{o \in Z_I}\,\olU_o$.  The sections $\mathbf{T}_{i_a} \in H^0(\sU_o,\left.\tGL(V_\bZ)\right|_{\sU_o})$ are independent of our choice of local coordinates, and so define flat sections $\{\mathbf{T}_{i_a}\}_{i_a\in I} \subset H^0(U,\left.\tGL(\mathbb{V})\right|_U)$ over $U = B \cap X$.
\end{proof}

\begin{remark} \label{R:cTI}
As in \S\ref{S:loclift}, the $k$-tuple of flat sections $\{\mathbf{T}_{i_a}\}_{i_a\in I} \subset H^0(U,\left.\tGL(\mathbb{V})\right|_U)$ determines a $\Gamma$-conjugacy class $\cT_I \subset \tGL(V_\bZ) \times \cdots \times \tGL(V_\bZ)$.  We have $\cT_I = \cT_{\sU_o}$ for all coordinate neighborhoods $\olU_o$ in the proof of Lemma \ref{L:nbd-ZI}.
\end{remark}

It follows from Lemmas \ref{L:good-quotients} and \ref{L:nbd-ZI} that the $\nu_I \circ F_I$ patch together to define the holomorphic map 
\[
  \Psi_I : Z_I^* \ \to \ (\Gamma_I\exp(\bC\s_I)) \bs \sM_I
\]
of \eqref{E:PsiI}.

\subsubsection{} \label{S:PhiI}

The quotient space 
\begin{equation}\label{E:DI}
  \sD_I \ = \ C_{I,\bC}^{-1} \bs \sM_I
\end{equation}
is a product (over $0 \le k \le \sfn$) of Mumford--Tate domains parameterizing weight $\sfn+k$ Hodge structures on $P_{\sfn+k}(N)$ that are polarized by $Q(\cdot , N^k \cdot )$, for any $N \in \s_I$.  Just as the group $\cG_I$ of \eqref{E:cGI} acts on $\sM_I$ (Remark \ref{R:GonM}), the group 
\begin{equation}\label{E:cLI}
  \cL_I \ = \ C_{I,\bC}^{-1} \bs \cG_I
  \ = \ C_{I,\bR}^{-1} \bs C_{I,\bR}
\end{equation}
acts transitively on $\sD_I$.  

The condition \eqref{E:W(N)} implies that $\exp(\bC\s_I) \subset C_{I,\bC}^{-2}$.  So the quotient map $\sM_I \to \sD_I$ descends to $\exp(\bC\s_I)\bs \sM_I \to \sD_I$, and we obtain a commuting diagram 
\begin{equation}\label{E:Psi2Phi}
  \begin{tikzcd}
    Z_I^* \arrow[r,"\Psi_I"]  \arrow[rd,"\Phi_I"']
    & (\Gamma_I \exp(\bC\s_I))\bs \sM_I \arrow[d]\\
    & \Gamma_I \bs \sD_I
  \end{tikzcd}
\end{equation}
that defines the map $\Phi_I$ of \eqref{E:PhiTheta}.

\subsubsection{} \label{S:ThetaI}
The parabolic subgroup stabilizing $W$ is filtered
\[
  P_W \ \supset \ P_W^{-1} \ \supset \ P_W^{-2} \ \supset \ P_W^{-3}
  \ \supset \cdots
\]
by normal subgroups
\begin{equation}\label{E:dfnPWa}
  P_W^{-a} \ = \ \{ g \in P_W \ | \ g \hbox{ acts trivially on }
  W_\ell/W_{\ell-a} \,,\ \forall \ \ell \} \,.
\end{equation}
Likewise
\begin{equation}\label{E:dfnCIa}
  C_I^{-a} \ = \ C_I \,\cap\, P_W^{-a}
\end{equation}
defines a filtration of $C_I$ by normal subgroups.  Setting 
\begin{equation}\label{E:MI1}
  \sM_I^1 \ = \ C_{I,\bC}^{-2} \bs \sM_I
\end{equation}
factors the quotient map 
\begin{equation}\label{E:p}
  p : \sM_I \ \to \ \sD_I
\end{equation}
as 
\begin{equation}\label{E:p1p0}
\begin{tikzcd} 
  \sM_I \arrow[r,"p_1"] 
  & \sM_I^1 \arrow[r,"p_0"] 
  & \sD_I \,.
\end{tikzcd}
\end{equation}
This in turn allows us to expand \eqref{E:Psi2Phi} to a tower
\begin{equation}\label{E:towerI*}
  \begin{tikzcd}
    Z_I^* \arrow[r,"\Psi_I"]  
    	\arrow[rd,"\Theta_I"',end anchor={west}]
    	\arrow[rdd,"\Phi_I"',end anchor={west},bend right]
    & (\Gamma_I \exp(\bC\s_I))\bs \sM_I 
		\arrow[d,"\pi_1"] \\
    & \Gamma_I \bs \sM_I^1 \arrow[d,"\pi_0"] \\
    & \Gamma_I \bs \sD_I
  \end{tikzcd}
\end{equation}
defining the map $\Theta_I$ of \eqref{E:PhiTheta}.

\subsection{Deligne splittings} 

\subsubsection{} 

Given a mixed Hodge structure $(W,F)$ on $V_\bQ$, we have a Deligne splitting
\begin{equation}\label{E:dsV}
  V_\bC \ = \ \op\, V^{p,q}_{W,F} 
\end{equation}
satisfying
\begin{equation} \label{E:ds-WF}
  W_\ell \ = \ \bigoplus_{p+q\le \ell} V^{p,q}_{W,F} \tand
  F^k \ = \ \bigoplus_{p\ge k} V^{p,q}_{W,F} \,,
\end{equation}
and
\begin{equation}\label{E:ds-conj}
  \overline{V^{p,q}_{W,F}} \ \equiv \ V^{q,p}_{W,F}
  \quad \hbox{mod} \quad \bigoplus_{r<p,q<s} V^{r,s}_{W,F} \,.
\end{equation}
It follows from the first equality in \eqref{E:ds-WF} that the restriction of the natural projection $W_{\ell,\bC} \to \tGr^W_{\ell,\bC} = W_{\ell,\bC}/W_{\ell-1,\bC}$ to $\oplus_{p+q=\ell}\, V^{p,q}_{W,F}$ is an isomorphism.  That is, we have a natural identification
\begin{equation}\label{E:GrWiso}
    \tGr^W_{\ell,\bC} \ \simeq \ 
    \bigoplus_{p+q=\ell} V^{p,q}_{W,F} \,.
\end{equation}
If $(W,F,N)$ is a limiting mixed Hodge structure, then 
\begin{equation}\label{E:NVpq}
  N(V^{p,q}_{W,F}) \ \subset \ V^{p-1,q-1}_{W,F} \,,
\end{equation}
and the map
\begin{equation}\label{E:NkVpq}
  N^k : V^{p,q}_{W,F} \ \to \ 
  V^{p-k,q-k}_{W,F}
\end{equation}
is an isomorphism, for all $p+q=n+k$.  The decomposition is $Q$-orthogonal in the sense that
\begin{equation}\label{E:ds-Q}
  Q(V^{p,q}_{W,F} \,,\, V^{r,s}_{W,F}) \ = \ 0 \,,
  \quad\hbox{for all}\quad (p+r,q+s) \not= (\sfn,\sfn) \,.
\end{equation}

\subsubsection{} \label{S:Rsplit}

The mixed Hodge structure $(W,F)$ is \emph{$\bR$--split} if equality holds in \eqref{E:ds-conj}.  Suppose that $(W,F,\s_I)$ is a limiting mixed Hodge structure.  Let 
\[
  \fc_I \ = \ \{ \xi \in \fg \ | \ [\xi,N] = 0 \,,\ N \in \s_I \}
\]
be the Lie algebra of $C_I$.  Then there exists $\d \in \fc_{I,\bR} \,\cap\, \bigoplus_{p,q\le-1}\,\fg^{p,q}_{W,F}$ so that $\tilde F = e^{\bi\d}\cdot F$ defines an $\bR$--split limiting mixed Hodge structure $(W, \tilde F , \s_I)$, \cite[(2.20)]{MR840721}.  

\subsubsection{} \label{S:induced}

A limiting mixed Hodge structure $(W,F,N)$ on $V$ induces one on the Lie algebra $\fg$.  The Hodge and weight filtrations are 
\begin{subequations}\label{SE:g-FW}
\begin{eqnarray}
  F^p(\fg) & = & \{ \x \in \fg_\bC \ | \ \xi(F^k) \subset F^{k+p} 
  \,,\ \forall \ k \} \\
  W_\ell(\fg) & = & \{ \xi \in \fg \ | \ 
  \xi(W_k) \subset W_{k+\ell} \,,\ \forall \ k \} \,.
\end{eqnarray}
\end{subequations}
The nilpotent operator is $\tad_N : \fg \to \fg$.  And the induced polarization $\sQ$ on $V \ot V^\vee = \tEnd(V) \supset \fg$ is given by 
\begin{equation}\label{E:dfn-sQ}
  Q(\xi_1(v_1) , \xi_2(v_2) )
  \ = \ Q(v_1,v_2)\,\sQ(\xi_1,\xi_2)
\end{equation}
for all $\xi_1,\xi_2 \in \tEnd(V)$ and $v_1,v_2 \in V$.  Note that that $\sQ$ is necessarily symmetric, and $\tAd(G)$--invariant.  The triple $(W(\fg),F(\fg),\tad_N)$ is a limiting mixed Hodge structure on $(\fg,\sQ)$.  

\subsubsection{}

The Deligne splitting 
\begin{subequations}\label{SE:gpq}
\begin{equation}
  \fg_\bC \ = \ \bigoplus\,\fg^{p,q}_{W,F} \,,
\end{equation}
for the induced MHS is given by 
\begin{equation} \label{E:gpq-dfn}
  \fg^{p,q}_{W,F} \ = \ 
  \{ \x \in \fg_\bC \ | \ \x(V^{r,s}_{W,F}) \subset V^{p+r,q+s}_{W,F}\,, 
  \ \forall \ r,s \} \,,
\end{equation}
and is compatible with the Lie bracket in the sense that 
\begin{equation}\label{E:lieb}
  [ \fg^{p,q}_{W,F} \,,\, \fg^{r,s}_{W,F} ] \ \subset \ \fg^{p+r,q+s}_{W,F} \,.
\end{equation}
\end{subequations}
The analog of \eqref{E:ds-Q} here is
\begin{equation}\label{E:ds-sQ}
  \sQ(\fg^{p,q}_{W,F} \,,\, \fg^{r,s}_{W,F}) \ = \ 0 \,,
  \quad\hbox{for all}\quad (p+r,q+s) \not= (0,0) \,.
\end{equation}
Equation \eqref{E:NVpq} implies
\begin{equation}\label{E:Nin}
  N \ \in \ \fg^{-1,-1}_{W,F} \,.
\end{equation}

\subsubsection{}
It follows from \eqref{SE:g-FW} that the Lie algebra of the parabolic subgroup $P_{W,\bC} \subset G_\bC$ preserving the weight filtration $W$ is 
\begin{equation}\label{E:pW}
  \fp_{W,\bC} \ = \ W_0(\fg_\bC) \ = \ 
  \bigoplus_{p+q\le0} \fg^{p,q}_{W,F} \,.
\end{equation}
Likewise 
\begin{equation}\label{E:F0g}
  \ff \ = \ F^0(\fg_\bC) \ = \ \bigoplus_{p\ge0}\,\fg^{p,q}_{W,F}
\end{equation}
is the parabolic Lie algebra of the stabilizer $\tStab_{G_\bC}(F)$.

Note that \eqref{E:Nin} implies
\begin{equation} \label{E:sIin}
  \s_I \ \subset \ \fg^{-1,-1}_{W,F} \,.
\end{equation} 
So $\fc_I$ inherits the Deligne splitting
\begin{equation}\label{E:ds-cI}
  \fc_{I,\bC} \ = \ 
  \bigoplus_{p+q\le0} \fc_{I,F}^{p,q} \,,\quad
  \fc_{I,F}^{p,q} \,=\, 
  \fc_{I,\bC} \cap \fg^{p,q}_{W,F} \,.
\end{equation}
And \eqref{E:dfnCIa} and \eqref{E:ds-WF} imply that the Lie algebra $C_{I,\bC}^{-a}$ is 
\begin{equation}\label{E:ds-cIa}
  \fc_{I,\bC}^{-a} \ = \ 
  \bigoplus_{p+q\le -a} \fc_{I,F}^{p,q} \,.
\end{equation}
From the definition of $\fc_I^a$ (or from \eqref{E:lieb} and \eqref{E:ds-cIa}) we see that
\begin{equation}\label{E:cIab}
  \left[ \fc_I^{-a} \,,\, \fc_I^{-b} \right]
  \ \subset \ \fc_I^{-a-b} \,.
\end{equation}

\subsection{Stabilizers and quotient representations} 

Since the parabolic subgroup $P_W$ preserves the weight filtration, it naturally acts on the graded quotients $\tGr^W_\ell = W_\ell/W_{\ell-1}$.  Let
\begin{equation} \label{E:rho}
  \varrho : P_W \ \to \ \bigoplus_\ell \tAut(\tGr^W_\ell)
\end{equation}
denote the induced representation.  By definition \eqref{E:dfnPWa}, the kernel of $\varrho$ is $P_W^{-1}$.  So the image $\varrho(P_W) \simeq P_W^{-1}\bs P_W$.  

The restriction of $\varrho$ to $C_I \subset P_W$ preserves the subspaces
\[
  P_{\sfn+k}(\s_I) \ = \ 
  \bigcap_{N \in \s_I} \tker\{ N^{k+1} : \tGr^W_{\sfn+k} \to \tGr^W_{\sfn-k-2} \} \,.
\]
This yields the representation
\begin{equation}\label{E:rhoI}
  \varrho_I \,=\, \left.\varrho\right|_{C_I} 
  : C_I \ \to \ \bigoplus_{k\ge0} P_{\sfn+k}(\s_I) \,.
\end{equation}
As above, the kernel of $\varrho_I$ is $C_I^{-1}$ by definition \eqref{E:dfnCIa}, and the image $\varrho_I(C_I)$ is isomorphic to the Levi quotient $C_I^{-1}\bs C_I$.  In particular, $\cL_I = \varrho_I(C_{I,\bR})$, \eqref{E:cLI}.

\begin{lemma} \label{L:cpt-stab-CI}
Let $(W,F,\s_I)$ be a limiting mixed Hodge structure, as in \S\ref{S:LMHS-F(w)}.  Let $\tStab_{C_{I,\bR}}(F)$ denote the stabilizer of $F \in \sM_I$ in $C_{I,\bR} \subset \cG_I$.  Then $\varrho_I(\tStab_{C_{I,\bR}}(F)) = \tStab_{\cL_I}(p(F))$ is compact.  Moreover, we have a natural isomorphism $\tStab_{C_{I,\bR}}(F) \simeq \varrho_I(\tStab_{C_{I,\bR}}(F))$ of real Lie groups.  In particular, $\tStab_{C_{I,\bR}}(F)$ is compact.
\end{lemma}

\begin{proof}
Let $N \in \s_I$.  The action \eqref{E:rhoI} preserves the polarization $Q(\cdot , N^k\cdot)$ on the primitive subspaces $P_{\sfn+k}(N)$.  And $\tStab_{C_{I,\bR}}(F)$ preserves the induced Hodge filtration $p(F)$ on $P_{\sfn+k}(N)$.  So $\varrho_I(\tStab_{C_{I,\bR}}(F))=\tStab_{\cL_I}(p(F))$ is the stabilizer of a polarized Hodge structure.  It follows that $\varrho_I(\tStab_{C_{I,\bR}}(F))$ is compact.

Keeping \eqref{E:CinP} in mind, we see from \eqref{E:ds-conj}, \eqref{E:pW} and \eqref{E:F0g} that the stabilizer of $F$ in $P_{W,\bR}^{-1} \supset C_{I,\bR}^{-1}$ is trivial.  It follows that $\tStab_{C_{I,\bR}}(F) \simeq \varrho_I(\tStab_{C_{I,\bR}}(F))$.   
\end{proof}

\begin{corollary} \label{C:stab-GammaI}
The stabilizer $\tStab_{\Gamma_I}(F)$ of $F$ in $\Gamma_I = \Gamma \cap C_{I,\bQ}$ is finite.
\end{corollary}

\begin{remark} \label{R:stab-p(F)}
Note that  
$
  \tStab_{C_{I,\bR}}(p(F)) \ \simeq \ 
  \tStab_{\cL_I}(p(F)) \ltimes
  C_{I,\bR}^{-1} \,.
$
\end{remark}

\subsection{$\fsl_2$--triples}  \label{S:sl2}

Given a mixed Hodge structure $(W,F)$, define $Y \in \tEnd(V_\bC) = V_\bC \otimes V_\bC^\vee$ by specifying that $Y$ act on $V^{p,q}_{W,F}$ by the eigenvalue $\sfn - (p+q)$.  If $(W,F,\s_I)$ is a limiting mixed Hodge structure, then \eqref{E:ds-Q} implies $Y \in \fg_\bC$.  And given any $N \in \s_I$, there exists a unique $M = M(N) \in \fg_\bC$ so that $\{ M , Y , N \} \subset \fg_\bC$ is an \emph{$\fsl_2$--triple}
\begin{equation}\label{E:triple}
  [M , N] \,=\, Y \,,\quad 
  [Y , M] \,=\, 2M \tand
  [N , Y] \,=\, 2N ,
\end{equation}
cf.~\cite[(2.8)]{MR840721}.  We have $Y \in \fg^{0,0}_{W,F}$, and
\begin{equation}\label{E:Min}
  M \ \in \ \fg^{1,1}_{W,F} \,.
\end{equation}


\begin{remark} \label{R:rescaleM}
The collection of all $\{ M = M(N) \ | \ N \in \s_I \}$ forms a cone: if we scale $N \mapsto k N$, then we can see from \eqref{E:triple} that $M$ scales to $\frac{1}{k} M$.
\end{remark}

\begin{remark} \label{R:Yratl}
Given a limiting mixed Hodge structure $(W,F,\s_I)$, there exists $\hat F \in C_{I,\bC}^{-1} \cdot F$ so that semisimple operator $Y = Y(W,\hat F)$ defined by $(W,\hat F)$ is rational $Y \in \fg_\bQ$, \cite[Lemma 3.2]{MR3474815}.  In this case, $M \in \fg_\bQ$ is also rational \cite[Proof of (5.17)]{MR0382272}.
\end{remark}

\begin{remark}
If $Y$ is rational, then $(W,F)$ is $\bR$--split, cf.~\S\ref{S:Rsplit}.
\end{remark}


\begin{remark} \label{R:samefibre}
Note that $F$ and $\hat F$ lie in the same fibre $C_{I,\bC}^{-1} \cdot F$ of the map $p : \sM_I \to \sD_I$ defined in \eqref{E:p}.
\end{remark}

Assume that $F \in \sM_I$ has been selected so that $Y$ is rational.  Then the centralizer
\begin{equation} \nonumber 
  L_I \ = \ \{ g \in C_I \ | \ \tAd_g(Y) = Y \}
\end{equation}
of $Y$ is a Levi factor of $C_I$.  In particular,
\begin{equation}\label{E:CIY}
  C_I \ = \ L_I 
  \,\ltimes\, C_I^{-1} \,,
\end{equation}
and the restriction of the quotient map $\varrho_I : C_I \to C_I^{-1} \bs C_I = \varrho_I(C_I)$ to $L_I$ is an isomorphism; in particular, $L_{I,\bR} \simeq \cL_I = \varrho_I(C_{I,\bR})$, \eqref{E:cLI}.  The group $\cG_I$ of \eqref{E:cGI} acting transitively on $\sM_I$ is
\begin{equation}\label{E:cGIY}
  \cG_I \ = \ L_{I,\bR} \,\ltimes\, C_{I,\bC}^{-1} \,.
\end{equation}

\begin{remark} \label{R:stabM}
Because $M$ is uniquely determined by $N,Y$, it is necessarily stabilized by $L_I$; that is $\tAd_g(M) = M$ for all $g \in L_I$.
\end{remark}


\subsection{Nilpotent subaglebras} \label{S:nilpotent}

It will be helpful, at several points in this paper, to recall two basic properties of nilpotent subalgebras $\fu \subset \fg$:
\begin{i_list}
\item \label{i:expisom}
The exponential map $\exp : \fu \to G$ is a bijection onto a unipotent subgroup $\exp(\fu) \subset G$.
\item \label{i:factor}
If $\fu = \fu_1 \oplus \fu_2$, with the $\fu_i$ subalgebras, then the product map $\exp(\fu_1) \times \exp(\fu_2) \to \exp(\fu)$ is a bijection.
\end{i_list}

\begin{remark} \label{R:exp(cIa)}
Recall the normal subgroups $C_I^{-a}$ of \eqref{E:dfnCIa}, $a \ge 1$.  Since $C_I^{-a}$ is unipotent, the exponential map $\exp : \fc_{I,\bC}^{-a} \to C_{I,\bC}^{-a}$ is a biholomorphism (item \ref{i:expisom} above).  
\end{remark}


\subsection{Schubert cells and period matrix representations} \label{S:schubert}

Let $\tPhi(t,w) = \exp(\sum\ell(t_i) N_i) g(t,w) \cdot F$ be any local lift of $\Phi$ (as in \S\ref{S:loclift}).  We have $\fg_\bC = \ff \op \ff^\perp$ with 
\begin{equation}\label{E:ffperp}
  \ff^\perp \ = \ \bigoplus_{p<0}\,\fg^{p,q}_{W,F}
\end{equation}
a nilpotent subalgebra of $\fg_\bC$.  The map $g$ of \eqref{E:tilde-g} is determined by specifying that it take value in $\exp(\ff^\perp) \subset G_\bC$: 
\begin{equation}\label{E:tilde-g2}
  g : \olU \ \to \ \exp(\ff^\perp) \,.
\end{equation}

\begin{remark} \label{R:logtg}
Since $\ff^\perp$ is nilpotent, \S\ref{S:nilpotent}.\ref{i:expisom} implies $\log g : \olU \to \ff^\perp$ is holomorphic.
\end{remark}

\begin{remark} \label{R:tg=0}
Replacing $F$ with $g(0,0) \cdot F$ if necessary, we may assume that $g(0,0) = \mathrm{Id}$.  (In general, it will not be possible to simultaneously normalize $g(0,0) = \mathrm{Id}$ and have $Y$ be rational as in Remark \ref{R:Yratl}.)
\end{remark}

Since \eqref{E:Nin} implies that the $N_i \in \ff^\perp$, we see that 
\begin{equation}\label{E:xi}
  \xi(t,w) \ = \ \exp( \tsum \ell(t_i) N_i ) \,g(t,w)
\end{equation}
takes value in $\exp(\ff^\perp)$.  It follows that the local lift $\tPhi(t,w) = \xi(t,w) \cdot F$ takes value in the open Schubert cell $\sS \subset \check \sD$
\begin{equation} \label{E:Sp}
  \sS \ = \ \exp(\ff^\perp) \cdot F 
  \ = \ \left\{ E \in \check \sD \ | \ 
  \tdim\,(E^a \cap \overline{F_\infty^b}) 
  = \tdim\,(F^a \cap \overline{F_\infty^b}) \,,\ 
  \forall \ a ,b \right\} \,,
\end{equation}
defined by
\begin{equation}\label{E:conjFinfty}
  \overline{F_\infty^b} \ = \ 
  \bigoplus_{c \le n-b} V^{c,a}_{W,F} \,.
\end{equation}

\subsubsection{The filtration $F_\infty$} 

From \eqref{E:ds-conj} and \eqref{E:conjFinfty} we see that 
\begin{equation}\label{E:ds-Finfty}
  F_\infty^b \ = \ 
  \bigoplus_{c \le n-b} V^{a,c}_{W,F} \,.
\end{equation}
In fact, \eqref{E:NVpq} and \eqref{E:NkVpq} can be used to show that 
\begin{equation}\label{E:Finfty}
  F_\infty \ = \ \lim_{y\to\infty}
  \exp(\bi y N) \cdot F
\end{equation}
for any $N \in \s_I$, cf.~\cite[\S2]{MR840721}.   

It follows from the second equation of \eqref{E:ds-WF}, \eqref{E:GrWiso} and \eqref{E:ds-Finfty} that $F$ and $F_\infty$ define the same filtration on $\tGr^W_\ell$.  Recalling the representation \eqref{E:rho}, this in turn implies
\begin{equation}\label{E:eqstab}
  \varrho(\tStab_{P_{W,\bC}}(F)) \ = \ 
  \varrho(\tStab_{P_{W,\bC}}(F_\infty))
\end{equation}

\subsubsection{The period matrix representation} \label{S:pmr}

The map $\ff^\perp \to \sS$ sending $x \mapsto \exp(x) \cdot F$ is a biholomorphism.  Let $\lambda : \sS \to \ff^\perp$ denote the inverse.  Then $\lambda \circ \tPhi(t,w)$ is the \emph{period matrix representation of $\Phi|_\sU$}.  Let $[\lambda \circ \tPhi(t,w)]^{p,q}$ be the component of 
\[
  \lambda \circ \tPhi(t,w) \ = \ 
  \log \xi(t,w) \ \in \ \ff^\perp
\]
taking value in $\fg^{p,q}_{W,F}$.  Then \eqref{E:tPhi}, \eqref{E:tilde-g}, \eqref{E:Nin}, \eqref{E:tilde-g2} and \eqref{E:xi} imply that:
\begin{i_list}
\item 
The component $[\lambda \circ \tPhi(t,w)]^{-1,q} = (\log g(t,w))^{-1,q}$, for all $q\not=-1$.
\item 
\label{i:11}
The component $[\lambda \circ \tPhi(t,w)]^{-1,-1} = \sum\ell(t_i)N_i  + (\log g(t,w))^{-1,-1}$.
\end{i_list}
The functions $(\log g)^{p,q} : \olU \to \fg^{p,q}_{W,F}$ are all holomorphic (Remark \ref{R:logtg}).

\subsection{Infinitesimal period relation} \label{S:ipr}

The variation of Hodge structure is subject to a differential constraint $\nabla \cF^p \subset \cF^{p-1} \ot \Omega^1_B$, known as \emph{infinitesimal period relation} (or \emph{Griffiths' transversality}).  

\subsubsection{The infinitesimal period relation over $\sU$}
The equivalent condition for the local lift $\widetilde\Phi(t,w) = \xi(t,w)\cdot F$ is 
\begin{subequations}\label{SE:ipr}
\begin{equation}
  (\xi^{-1} \td \xi)^{p,q} \ = \ 0 \,,\quad\forall \ p \le -2 \,.
\end{equation}
Here $\xi^{-1} \td \xi$ is the pull-back of the Maurer-Cartan form on the Lie group $\exp(\ff^\perp)$ under the map $\xi : \olU \to \exp(\ff^\perp)$.  This 1-form takes value in the Lie algebra $\ff^\perp$, and $(\xi^{-1} \td \xi)^{p,q}$ is the component taking value in $\fg^{p,q}_{W,F} \subset \ff^\perp$.  The infinitesimal period relation
\begin{equation}
  \xi^{-1} \td \xi \ = \ \bigoplus_q (\xi^{-1} \td \xi)^{-1,q}
\end{equation}
implies that the infinitesimal variation in the period map is encoded by the horizontal data in the Maurer-Cartan form of $\exp(\ff^\perp)$.  It follows from \eqref{E:gpq-dfn} and \eqref{E:F0g} that $(\xi^{-1} \td \xi)^{-1,q} = \td \xi^{-1,q}$, so that 
\begin{equation}
  \xi^{-1} \td \xi \ = \ \bigoplus_q \td \xi^{-1,q} \,.
\end{equation}
\end{subequations}



\subsubsection{The infinitesimal period relation over $Z_I^* \cap \olU$}

Along $Z_I^* \cap \olU$ the definition \eqref{E:xi} of $\xi$ and the infinitesimal period relation \eqref{SE:ipr} force the map $g$ of \eqref{E:tPhi}, \eqref{E:tilde-g} and \eqref{E:tilde-g2} to take value in the centralizer of $\s_I$,
\begin{equation}\label{E:tilde-g3}
  g(0,w) \ \in \ 
  C_{I,\bC} \, \cap \, \exp(\ff^\perp) 
  \ = \ \exp(\fc_{I,\bC}\,\cap\,\ff^\perp)\,,
\end{equation}
and to satisfy the differential constraint $\left.(g^{-1} \td g)^{p,q}\right|_{Z_I^* \cap \olU} = 0$ for all $p \le -2$.  Just as in \eqref{SE:ipr}, this is equivalent to 
\begin{equation}\label{E:ipr-tg}
  \left.g{}^{-1} \td g \right|_{Z_I^* \cap \olU}
  \ = \ 
  \bigoplus_q \left. 
  \td g{}^{-1,q} \right|_{Z_I^* \cap \olU}\,.
\end{equation}

\begin{remark}\label{R:ipr-tg}
In particular, if $\left.\td g^{-1,q}\right|_{Z_I^* \cap \olU} = 0$ for all $q$, then $\left.\td g\right|_{Z_I^* \cap \olU} = 0$.  
\end{remark}

\begin{remark} \label{R:ipr-tg*}
Note that \eqref{E:ds-cI} and \eqref{E:tilde-g3} imply
\[
    \left. (g^{-1} \td g)^{p,q}\right|_{Z_I^* \cap \olU} \ = \ 0 \,,\quad \forall \quad p+q \ge 1 \,.
\]
This allows us to refine Remark \ref{R:ipr-tg} as follows.  Fix $-q \le 0$.  If $\left.\td g^{-1,-r}\right|_{Z_I^* \cap \olU} = 0$ for all $-q \le -r \le 0$, then $\left.\td g^{-p,-r}\right|_{Z_I^* \cap \olU} = 0$ for all $-p \le -1$ and all $-q\le -r \le 0$. 
\end{remark}

\section{Proper extensions}

The maps $\Psi_I, \Theta_I$ and $\Phi_I$ of \eqref{E:towerI*} naturally admit maximal holomorphic extensions, and these extensions are proper.  This implies that the fibres are compact.  The compactness of the $\Theta_I$--fibres will play a key role in the proof of Theorem \ref{T:thatsall}.

\subsection{Proper extension of $\Psi_I$} 

Recall that the $Z_I$ are connected by assumption (\S\ref{S:nilp-orb-infty}.  This is inessential, but notationally convenient).  Let $X_I \subset \olB$ be the neighborhood of $Z_I$ given by Lemma \ref{L:nbd-ZI}, and let $\{\mathbf{T}_{i_a}\}_{i_a\in I} \subset H^0(U_I , \left.\tGL(\mathbb{V})\right|_{U_I})$  be the flat sections over $U_I = B \cap X_I$ that were constructed in the proof of the lemma.  By assumption the $\{\mathbf{T}_{i_a}(u) \ | \ u \in U_I \}_{i_a\in I}$ are unipotent operators (\S\ref{S:setup}); let $\{\mathbf{N}_{i_a} = \log \mathbf{T}_{i_a}\} \subset H^0(U_I,\left.\fgl(\mathbb{V} \ot_\bZ \bQ)\right|_{U_I})$ be the logarithms.  Let
\[
  \mathbf{S}_I \ \subset \ 
  \left.\fgl(\mathbb{V} \ot_\bZ \bQ)\right|_{U_I}
\]
be the flat sub-bundle over $U_I$ point-wise spanned by the $\{\mathbf{N}_{i_a}\}_{i_a\in I}$.  Without loss of generality $X_J \subset X_I$ and $U_J \subset U_I$ whenever $I \subset J$.  Then $\left.\mathbf{S}_I\right|_{U_J}$ is a flat subbundle of $\mathbf{S}_J$.  Let $\Upsilon$ be the finest possible partition of the collection 
\[
  \cI \ = \ \{ I \ | \ Z_I^* \not= \emptyset\}
\] 
that satisfies the following property: if $I \subset J$ and $\left.\mathbf{S}_I\right|_{U_J} = \left.\mathbf{S}_J\right|_{U_J}$, then $I \sim J$.  Fix $\upsilon \in \Upsilon$, and define 
\[
  \Zc_\upsilon \ = \ \bigcup_{I \in \upsilon} Z_I^* \,.
\]
The quasi-projective intersection 
\[
  \Zc_I \ = \ Z_I \cap \Zc_\upsilon
\]
is the \emph{cone-closure} of $Z_I^*$.

Given $0 < y_a \in \bQ$, consider the section
\[
  \mathbf{N}_I(y) \ = \ \sum_{a=1}^k y_a \,\mathbf{N}_{i_a}
  \ \in \ H^0(U_I,\left.\fgl(\mathbb{V}\ot_\bZ\bQ)\right|_{U_I}) \,.
\]
As in \S\ref{S:W}, the section $\mathbf{N}_I(y)$ point-wise defines a filtration
\[
  \mathbf{W}^I_0 \,\subset\,
  \mathbf{W}^I_1 \,\subset \cdots \subset\,
  \mathbf{W}^I_{2\sfn} \,=\, 
  \left.\mathbb{V} \ot_\bZ \bQ\right|_{U_I}
\]
of $\left.\mathbb{V} \ot_\bZ \bQ\right|_{U_I}$ by flat subbundles.  The filtration is independent of our choice of $y = (y_a) \in \bQ^k_+$.  Define 
\[
  \mathrm{wt}(I) \ = \ 
  \left\{ J \in \cI \ | \ J \supset I \,,\ \left.\mathbf{W}^I_\bullet\right|_{U_I \cap U_J} = \left.\mathbf{W}^J_\bullet\right|_{U_I \cap U_J} \right\}\,.  
\]

Let $\Phi_{U_I} : U_I \to \Gamma_I \bs \sD$ be the period map of Lemma \ref{L:nbd-ZI}.  Recall that $\Phi_{U_I}$ was defined by first fixing an element $\{T_i\}_{i\in I}$ in the $\Gamma$--conjugacy class $\cT_\sU$.  So $N_i = \log T_i$ is well-defined for all $i \in I$.  If $J \supset I$, then the flat sections $\{\mathbf{N}_j\}_{j\in J}$ over $X_J \cap U_I = U_J$ determine nilpotent operators $\{N_j\}_{j \in J} \in \fgl(V_\bQ)$ that are well-defined up to the action of $\Gamma_I$.  So the $\Gamma_I$--conjugacy class of $\Gamma_J \subset \Gamma_I$ is well-defined, and the $\Gamma_I$--congruence class of $\sM_J \subset \check D$ is well-defined.

\begin{lemma}\label{L:MJinMI}
Given $J \in\mathrm{wt}(I)$, there are natural maps 
\begin{equation} \nonumber 
  \Gamma_J \bs \sM_J \,\to\, \Gamma_I\bs \sM_I \,,\quad
  \Gamma_J \bs \sM_J^1 \,\to\, \Gamma_I \bs \sM_I^1 \tand
  \Gamma_J \bs \sD_J \,\to\, \Gamma_I \bs \sD_I
\end{equation}
with well-defined images.
\end{lemma}

\begin{proof}
Since tuple $\{N_j\}_{j\in J} \subset \fgl(V_\bQ)$ is well-defined up to the action of $\Gamma_I$, it follows that the $\Gamma_I$--conjugacy class of the centralizer $C_J \subset C_I$ is well-defined.  Likewise, $\Gamma_I(\sM_J) \subset \check D$ is well-defined.  It is a nontrivial result \cite[(B.19)]{GGR-part1} that
\begin{equation}\label{E:MJinMI}
  \Gamma_I(\sM_J) \ \subset \ \sM_I \,.
\end{equation}
This yields the first map $\Gamma_J \bs \sM_J \to \Gamma_I \bs \sM_I$ of the lemma.

We also have \cite[Remark B.20]{GGR-part1}
\begin{equation}\label{E:CJvCI}
  C_J^{-1} = C_J \cap C_I^{-1} \,.
\end{equation}
Together \eqref{E:MJinMI} and \eqref{E:CJvCI} define the third map $\Gamma_J \bs \sD_J \to \Gamma_I \bs \sD_I$ of the lemma.  It follows from \eqref{E:CJvCI}, and the definition of $C_I^{-a}$ in \eqref{E:dfnCIa}, that $C_J^{-2} = C_J \cap C_I^{-2}$.  Keeping \eqref{E:MJinMI} in mind, this yields the second map $\Gamma_J \bs M_J^1 \inj \Gamma_I \bs \sM_I^1$.
\end{proof}

\begin{lemma} \label{L:Wup}
If $I \subset J$ and $I \sim J$, then $J \in \mathrm{wt}(I)$.  In particular, $\Zc_I \subset \Zw_I$.
\end{lemma}

\begin{proof}
The conditions $I \subset J$ and $\left.\mathbf{S}_I\right|_{U_I \cap U_J} = \left.\mathbf{S}_J\right|_{U_I \cap U_J}$ imply that the weight filtrations coincide, \cite[Lemma B.10]{GGR-part1}. 
\end{proof}

From the local coordinate expression \eqref{E:FI} for $\Psi_I$, and Lemmas \ref{L:nbd-ZI}, \ref{L:MJinMI} and \ref{L:Wup} we deduce

\begin{corollary} 
The map $\Psi_I : Z_I^* \to (\Gamma_I \exp(\bC\s_I)) \bs \sM_I$ extends to the cone closure $\Zc_I$.
\end{corollary}

\noindent In a mild abuse of notation, we will also denote the extension by $\Psi_I$.

\begin{lemma}[{\cite{DRstein}}]\label{L:extn2Zc}
The extension $\Psi_I : \Zc_I \to (\Gamma_I \exp(\bC\s_I)) \bs \sM_I$ is proper.
\end{lemma}

\subsection{Proper extensions of $\Theta_I$ and $\Phi_I$} \label{S:Zw}

While the map $\Psi_I$ is proper on $\Zc_I$, the map $\Theta_I$ need not be.  We may obtain a proper extension as follows.

\begin{lemma}
The union
$\displaystyle
  \Zw_I \ = \ \bigcup_{J \in \mathrm{wt}(I)} Z_J^*$
is quasi-projective.
\end{lemma}

\begin{proof}
Given $J \in \mathrm{wt}$ and $I \subset J' \subset J$, it suffices to show that $J' \in \mathrm{wt}(I)$.  This is \cite[Corollary B.11(a)]{GGR-part1}
\end{proof}

\begin{lemma} \label{L:proper2}
The maps $\Theta_I$ and $\Phi_I$ of \eqref{E:towerI*} admit holomorphic, proper extensions to $\Zw_I \supset \Zc_I$
\begin{equation}\label{E:towerIwt}
  \begin{tikzcd}
    \Zw_I \arrow[r,"\Theta_I"]
    	\arrow[rd,"\Phi_I"',end anchor={north west}]
    & \Gamma_I \bs \sM_I^1 \arrow[d,"\pi_0"] \\
    & \ \Gamma_I \bs \sD_I \,.
  \end{tikzcd}
\end{equation}
\end{lemma}

\begin{proof}
Holomorphicity and properness of $\Phi_I : \Zw_I \to \Gamma_I \bs \sD_I$ follows directly from \cite[Lemma B.1]{GGR-part1}.  Holomorphicity of $\Theta_I : \Zw_I \to \Gamma_I \bs \sM_I^1$ follows by essentially the same argument.  The key point in adapting the proof is that the $\exp(\bC\s_J)$ lie in $C_I^{-2}$ for every $J \in \mathrm{wt}(I)$.  Then properness of $\Theta_I$ follows from properness of $\Phi_I$.
\end{proof}

\section{The relationship between $\Theta_I$ and $\Phi_I$} \label{S:Theta-v-Phi}

The goal of this section is to study the relationship between the maps $\Theta_I$ and $\Phi_I$ of \eqref{E:towerIwt}.  This relationship is given by 

\begin{theorem}\label{T:imageThetaI(A)}
The fibres of $\pi_0 : \Gamma_I \bs \sM_I^1 \to \Gamma_I \bs \sD_I$ are finite quotients of complex tori.  Each torus contains an abelian variety.  Let $A \subset \Zw_I$ be a connected component of a $\Phi_I$--fibre  \emph{(cf.~\eqref{E:towerIwt}.  In particular, $\Theta_I(A)$ is contained in a $\pi_0$--fibre).}  Then $\Theta_I(A)$ is contained in a (finite quotient of a) translate of the abelian variety.  The finite quotients are trivial if $\Gamma$ is neat.
\end{theorem}

\noindent The remainder of \S\ref{S:Theta-v-Phi} is devoted to the proof of Theorem \ref{T:imageThetaI(A)}: cf.~Corollary \ref{C:pi0-fibre}, Lemma \ref{L:sJ} and Lemma \ref{L:ThetaI(A)}.

\subsection{Fibre bundle structure of $p : \sM_I \to \sD_I$} \label{S:p}

The quotient map $p:\sM_I \to \sD_I$ is a fibre bundle.  We begin by reviewing this structure.

Fix an limiting mixed Hodge structure $F \in \sM_I$ so that the semisimple operator $Y$ determined by $(W,F)$ in \S\ref{S:sl2} is rational.  We have  
\[
  \sM_I \ = \ L_{I,\bR} \cdot 
  (C_{I,\bC}^{-1} \cdot F) \,,
\]
cf.~Remark \ref{R:GonM} and \eqref{E:cGIY}.  

\begin{remark} \label{R:CInegF}
It follows from the definitions \eqref{E:F0g}, \eqref{E:ds-cI} and \eqref{E:ffperp} that $\fc_{I,\bC}^{-1}$ decomposes as a direct sum
\[
  \fc_{I,\bC}^{-1} \ = \ 
  (\fc_{I,\bC}^{-1} \,\cap\,\ff^\perp)
  \ \op \ 
  (\fc_{I,\bC}^{-1} \,\cap\, \ff)
\]
of Lie subalgebras.  Then \S\ref{S:nilpotent}\ref{i:factor} yields
\[
  C_{I,\bC}^{-1} \ = \ 
  \exp(\fc_{I,\bC}^{-1} \,\cap\, \ff^\perp)
  \, 
  \exp(\fc_{I,\bC}^{-1} \,\cap\, \ff) \,.
\]
So $x \mapsto \exp(x) \cdot F$ defines a biholomorphism
\begin{equation}\label{E:expx}
  \fc_{I,\bC}^{-1} \,\cap\, \ff^\perp 
  \ \stackrel{\simeq}{\longrightarrow} 
  \ C_{I,\bC}^{-1} \cdot F \,.
\end{equation}    
\end{remark}

Remark \ref{R:CInegF} implies that $(g,x) \mapsto g \exp(x) \cdot F$ defines a surjection 
\begin{equation} \label{E:zetaMI}
  \zeta : L_{I,\bR} \,\times\, 
  (\fc_{I,\bC}^{-1} \,\cap\, \ff^\perp)
  \ \to \ \sM_I \,.
\end{equation}

\begin{lemma} \label{L:paramMI}
Given $(g,x) , (g',x') \in L_{I,\bR} \,\times\, 
(\fc_{I,\bC}^{-1} \,\cap\, \ff^\perp)$, we have $\zeta(g,x) = \zeta(g',x')$ if and only if $\tAd_gx = \tAd_{g'}x'$ and $g^{-1} g \in \tStab_{L_{I,\bR}}(F)$.
\end{lemma}

\begin{proof}
If $\zeta(g,x) = \zeta(g',x')$, then $p\circ\zeta(g,x) = p\circ\zeta(g',x')$.  By construction $p\circ\zeta(g,x) = g \cdot p(F)$, cf.~\eqref{E:p} and \eqref{E:cGIY}.  So $g^{-1} g' \in L_{I,\bR}$ stabilizes $p(F)$.  Then Lemma \ref{L:cpt-stab-CI} implies that $g^{-1}g'$ stabilizes $F$.

This implies $\zeta(g,x) = \zeta(g',x')$ if and only if $\exp(x) \cdot F = \exp( \tAd_{g^{-1} g'} x') \cdot F$.  The stabilizer $\tStab_{L_{I,\bR}}(F)$ of $F$ in $L_{I,\bR}$ preserves the Deligne splitting \eqref{E:dsV}; in particular, $\tAd_{g^{-1}g'}$ preserves $\fc_{I,\bC}^{-1} \,\cap\, \ff^\perp$.  So $\exp(x) \cdot F = \exp( \tAd_{g^{-1} g'} x') \cdot F$ holds if and only if $x = \tAd_{g^{-1} g'} x'$, cf.~\eqref{E:expx}.
\end{proof}

\begin{lemma} \label{L:fibre-p}
Given $\zeta(g,x) = g \exp(x) \cdot F \in \sM_I$, we have $p(g \exp(x) \cdot F) = g \cdot p(F) \in \sD_I$.  And the $p$--fibre over $g \cdot p(F) \in \sD_I$ is biholomorphic to $\tAd_g(\fc_{I,\bC}^{-1} \,\cap\, \ff^\perp)$.
\end{lemma}

\begin{proof}
We have $g \exp(x) \cdot F = \exp(\tAd_gx) g \cdot F$.  Since $C_I^{-1}$ is a normal subgroup of $C_I$, we have $\tAd_gx \in \fc_{I,\bC}^{-1}$ and $\exp(\tAd_gx) \in C_{I,\bC}^{-1}$.  So $p(g \exp(x) \cdot F) = g \cdot F$ follows directly from the definition \eqref{E:p}.

As in \eqref{E:expx}, the map $\tAd_g(\fc_{I,\bC}^{-1} \cap \ff^\perp) \to C_{I,\bC}^{-1} \cdot (gF)$ sending $\tAd_g(x) \mapsto \exp(\tAd_g(x)) \cdot (gF)$ is a biholomorphism.  This yields the identification of the fibre with $\tAd_g(\fc_{I,\bC}^{-1} \,\cap\, \ff^\perp)$.
\end{proof}

We note the following corollary of Lemma \ref{L:paramMI}.

\begin{corollary}
We have a commutative diagram
\[ \begin{tikzcd}
  L_{I,\bR} \,\times\, (\fc_{I,\bC}^{-1} \,\cap\,\ff^\perp) \arrow[r,"\zeta"] \arrow[d]
  & \sM_I \arrow[d]\\
  \displaystyle
  L_{I,\bR} \,\times\, \left(\frac{\fc_{I,\bC}^{-1} \,\cap\,\ff^\perp}{\bC\s_I}\right)
  \arrow[r,"\bar\zeta"]
  & \exp(\bC\s_I)\bs \sM_I \,,
\end{tikzcd} \]
and $\bar\zeta(g,\bar x) = \bar\zeta(g',\bar x')$ if and only if $\tAd_g\bar x = \tAd_{g'}\bar x'$ and $g^{-1} g \in \tStab_{L_{I,\bR}}(F)$.
\end{corollary}

We close this section with the 

\begin{proof}[Proof of Lemma \ref{L:good-quotients}]
Given $F \in \sM_I$, the orbit $C_{I,\bC}\cdot F$ is a complex submanifold of $\check \sD$.  The set $\sM_I$ is an open subset of this orbit, and so naturally a complex manifold \cite[\S3.5]{MR4742809}.  The action of $C_{I,\bR}$ on $\sM_I$ is proper \cite[Proposition 3.7]{MR4742809}.  Since $\Gamma_I \subset C_{I,\bR}$ is discrete, it follows that $\Gamma_I \bs \sM_I$ canonically admits the structure of a complex analytic space so that the quotient map $\sM_I \to \Gamma_I \bs \sM_I$ is holomorphic, \cite{MR84174}; see also \cite{Riemenschneider}.

It is a corollary of \eqref{E:sIin} and Lemma \ref{L:paramMI} that the action of $\exp(\bC\s_I) \subset \cG_I$ on $\sM_I$ is free and properly discontinuous.  It follows that $\exp(\bC\s_I) \bs \sM_I$ canonically admits the structure of a complex analytic manifold so that the quotient map $\sM_I \to \exp(\bC\s_I) \bs \sM_I$ is holomorphic.  

It remains to show that $\exp(\bC\s_I) \bs 
\sM_I \to (\Gamma_I \exp(\bC\s_I)) \bs \sM_I$ is a  morphism of complex analytic spaces.  Since $(\Gamma_I \cap \exp(\bC\s_I)) \bs \Gamma_I$ is a discrete subgroup of $\exp(\bR\s_I) \bs C_{I,\bR}$, it suffices to show that the action of $\exp(\bR\s_I) \bs C_{I,\bR}$ on $\exp(\bC\s_I)\bs \sM_I$ is proper. Any compact set $\overline K \subset \exp(\bC\s_I) \bs \sM_I$ can be realized (non-uniquely) as the image of a compact $K \subset \sM_I$.  (One way to do this is to fix a direct sum decomposition $\fc_{I,\bC}^{-1} \cap \ff^\perp = \bC\s_I \op \mathfrak{s}$.  The image $\zeta(L_{I,\bR} \times \mathfrak{s}) \subset \sM_I$ of the map \eqref{E:zetaMI} maps bijectively onto $\exp(\bC\s_I) \bs \sM_I$ under the projection $\sM_I \to \exp(\bC\s_I) \bs \sM_I$.  Given $\overline K \subset \exp(\bC\s_I) \bs \sM_I$ this bijection determines $K \subset \zeta(L_{I,\bR} \times \mathfrak{s}$.)  The the properness of the action of $\exp(\bR\s_I)\bs C_{I,\bR}$ on $\exp(\bC\s_I) \bs \sM_I$ follows from the properness of the action of $C_{I,\bR}$ on $\sM_I$.
\end{proof}

\subsection{Line bundles $\Lambda_M$ over $\Gamma_I \bs \sM_I$}

Fix $N  \in \s_I$, and let $\{M,Y,N\}$ be the $\fsl_2$--triple of \S\ref{S:sl2}.  Recall the induced bilinear form $\sQ$ on $\fg$ of \eqref{E:dfn-sQ}.  Define 
\[
  f_M' : L_{I,\bR} \ltimes (\fc_{I,\bC}^{-1} \,\cap\,\ff^\perp) \ \to \ \bC^* \,=\,\bC \bs \{0\}
  \quad \hbox{by} \quad 
  f_M'(g,x) \ = \ 
  \exp 2 \pi \bi \sQ(M,\tAd_g(x)) \,.
\]
Lemma \ref{L:paramMI} implies that $f_M'$ induces a well-defined map
\[
  f_M'' : \sM_I \ \to \ \bC^* \,.
\]
In fact, the $\tAd(G)$--invariance of $\sQ$ (\S\ref{S:induced}) and the $\tAd(L_{I,\bR})$--invariance of $M$ (Remark \ref{R:stabM}) allow us write
\begin{equation}\label{E:invar}
  \sQ(M , \tAd_g y) \ = \ 
  \sQ( \tAd_g^{-1} M , y) \ = \ 
  \sQ(M,y) \quad \hbox{ for all } \quad
  g \in L_{I,\bR} \,,\ y \in \fg_\bC \,.
\end{equation}
In particular, we have
\begin{equation} \nonumber 
  f_M'(g,x) 
  \ = \ \exp 2 \pi \bi \sQ(\tAd_g^{-1}M,x)
  \ = \ \exp 2 \pi \bi \sQ(M,x) \,.
\end{equation}

Define
\[
  f_M : \Gamma_I \times \sM_I \ \to \ \bC^*
  \quad\hbox{by}\quad
  f_M(\gamma,\zeta) \ = \ f_M''(\gamma \cdot \zeta)\,,
\]
and
\begin{equation}\label{E:eM}
  e_M : \Gamma_I \times \sM_I \ \to \ \bC^* 
  \quad\hbox{by} \quad 
  e_M(\gamma,\zeta) \ = \   
  \frac{f_M( \gamma , \zeta )}{f_M(1,\zeta)}\,.
\end{equation}
Then $e_M(\gamma_1\gamma_2,\zeta) = e_M(\gamma_1,\gamma_2 \cdot \zeta) \,e_M(\gamma_2,\zeta)$.  That is, $e_M$ is a factor of automorphy, and defines a line bundle over $\Gamma_I\bs \sM_I$.  Define a left action of $\Gamma_I$ on $\bC \times \sM_I$ by $\gamma\cdot(z,\zeta) = (z e_M(\gamma,\zeta) , \gamma \cdot \zeta)$.  Let
\begin{equation}\label{E:LM}
  \begin{tikzcd}[column sep=tiny]
      \Lambda_M \arrow[r,equal] \arrow[d] 
      & (\bC \times \sM_I)/\sim\\
      \Gamma_I \bs \sM_I
  \end{tikzcd}
\end{equation}
be the associated line bundle.  Then $f_M''(\zeta) = f_M(1,\zeta)$ defines a section $\psi_M : \Gamma_I \bs \sM_I \to \Lambda_M$.

\subsection{Action of $\Gamma_I$}\label{S:GammaIact}

Fix $(g,x) \in L_{I,\bR} \times (\fc_{I,\bC}^{-1} \cap \ff^\perp)$, and consider the action of $\gamma \in \Gamma_I$ on $\zeta(g,x) = g \exp(x)  \cdot F \in \sM_I$.  We have $\gamma \cdot \zeta(g,x) = \zeta(h,y) = h \exp(y) \cdot F$ for some $(h,y) \in L_{I,\bR} \times (\fc_{I,\bC}^{-1} \cap \ff^\perp)$.  In the subsequent sections we will need to know something about the relationship between $\gamma, (g,x)$ and $(h,y)$.  To begin, we need to factor $\gamma \in \Gamma_I \subset C_{I,\bR}$ with respect to the decomposition $C_{I,\bR} = L_{I,\bR} \ltimes C_{I,\bR}^{-1}$ of \eqref{E:CIY}.  For this, we assume that the Hodge filtration $F \in \sM_I$ has been chosen so that the triples $\{M,Y,N\}$ are rational (\S\ref{S:sl2}).  


Write $\gamma = \a \exp(b)$ with respect to the decomposition $C_I = L_I \ltimes C_I^{-1}$ of \eqref{E:CIY}; here $\a \in L_{I,\bQ}$ and $b \in \fc_{I,\bQ}^{-1}$.  We have 
\[
  \gamma \cdot \zeta(g,x) \ = \ 
  \a g \exp(\tAd_g^{-1}b) \exp(x) \cdot F \,.
\]
So $h = \a g$.  In general, $y$ is a complicated function of $\tAd_g^{-1}b \in \fc_{I,\bC}^{-1}$ and $x \in \fc_{I,\bC}^{-1} \cap \ff^\perp$ that is obtained by solving 
\begin{equation}\label{E:solvey}
  \exp(\tAd_g^{-1}b) \exp(x) \cdot F 
  \ = \ \exp(y) \cdot F
\end{equation}
for $y \in \fc_{I,\bC}^{-1} \cap \ff^\perp$.  Let $y^{p,q}$ be the component of $y$ taking value in $\fc_{I,F}^{p,q}$.  Keeping \eqref{E:lieb} and \eqref{E:ds-cIa} in mind, it is straightforward to work out 
\begin{equation}\label{E:y1}
  y^{-p,p-1} \ = \ 
  x^{-p,p-1} + (\tAd_g^{-1}b)^{-p,p-1} \,,
  \quad \hbox{ for all } \quad p > 0 \,.
\end{equation}

The components $y^{p,q}$, with $p+q \le -2$ are more difficult to work out.  We will have need of only $y^{-1,-1}$.  It is a more tedious exericse, by a still relatively straightfoward computation, to work out 
\begin{equation}\label{E:y11}
  y^{-1,-1} \ = \
  \left( x + \tAd_g^{-1}b + 
    \half [ \tAd_g^{-1}b , x] \right)^{-1,-1} \,.
\end{equation}
(For this, it is helpful to rewrite \eqref{E:solvey} as 
$\exp(x) \exp(\tAd_{\exp(x)}^{-1} \tAd_g^{-1}b) \cdot F = \exp(y) \cdot F$.)
Then \eqref{E:ds-sQ}, \eqref{E:Min}, \eqref{E:eM} and \eqref{E:y11} imply that 
\begin{equation}\label{E:eM2}
  e_M(\gamma,\zeta(g,x)) \ = \ 
  \exp 2 \pi \bi 
  \sQ \left(M \,,\, 
    \tAd_g^{-1}b + 
    \half [ \tAd_g^{-1}b , x] \right)\,.
\end{equation}

\begin{lemma} \label{L:Qint}
We may scale $M \in \fg_\bQ$ so that $\exp 2 \pi \bi \sQ(M,\tAd_g^{-1}b) = 1$ for all $\gamma = \a \exp(b) \in \Gamma_I$.
\end{lemma}

\begin{proof}
First, recall that a rescaling of $N$ induces a reciprocal rescaling of $M$ (Remark \ref{R:rescaleM}).  To prove the lemma, it suffices to show that we can scale $M$ so that $\sQ(M,\tAd_g^{-1}b) \in \bZ$ for all $\gamma = \a \exp(b) \in \Gamma_I$.  By \eqref{E:invar}, it suffices to show that we can scale $M$ so that  $\sQ(M,b) \in \bZ$ for all $\gamma = \a \exp(b) \in \Gamma_I$.  

The unipotent radical $\Gamma_I^{-1} = \Gamma_I \cap C_{I,\bQ}^{-1}$ of $\Gamma_I$ is an arithmetic subgroup of $C_I^{-1}$.  The decomposition $C_I = L_I \ltimes C_I^{-1}$ defines a projection $\varrho : C_I \to L_I$.  This projection is a morphism of $\bQ$--algebric groups.  So the image $\Gamma_I^0 \subset L_{I,\bQ}$ is an arithmetic group \cite[Theorem 1.2]{MR0204533}.  The product $\Gamma_I^0 \cdot \Gamma_I^{-1}$ is then an arithmetic subgroup of $C_I$, \cite[p.~20]{MR0204533}.  In particular, this product is commensurable with $\Gamma$.  So it suffices to prove the lemma for elements $\a\cdot\exp(b)$ of the product $\Gamma_I^0 \cdot \Gamma_I^{-1}$; that is, it suffices to prove the lemma for $\gamma = \exp(b) \in \Gamma_I^{-1}$.

A priori we have $\sQ(M,b) \in \bQ$.  Since $\Gamma_I^{-1}$ is an arithmetic group, and arithmetic groups are finitely generated \cite{MR0141670}, there exists $0 < k \in \bZ$ so that $\sQ(kM,b) \in \bZ$ for all $\gamma = \a\exp(b) \in \Gamma_I$.
\end{proof}

\begin{remark}\label{R:normM}
From this point on we restrict to $\fsl_2$--triples $\{N,Y,M\}$ with $M$ satisfying $\sQ(M,b) \in \bZ$ for all $\gamma = \a \exp(b) \in \Gamma_I$.  (Thanks to Lemma \ref{L:Qint} is this no real restriction, we need only rescale $N \in \s_I$.)  Then \eqref{E:invar} and \eqref{E:eM2} yield
\begin{equation}\label{E:eM3}
  e_M(\gamma,\zeta(g,x)) \ = \ 
  \exp 2 \pi \bi \sQ 
  \left(M \,,\, \half [ \tAd_g^{-1}b , x] \right)
  \ = \ 
  \exp \pi \bi \sQ 
  \left(M \,,\, [ b , \tAd_g x] \right)\,.
\end{equation}
\end{remark}

\subsection{A metric on $\Lambda_M$}

Define $h_M : L_{I,\bR} \ltimes (\fc_{I,\bC}^{-1} \,\cap\,\ff^\perp) \to \bR$ by 
\[
  h_M(g,x) \ = \ \exp \pi \bi \sQ(M\,,\,[\tAd_g x , \tAd_g \overline x]) \,.
\]
Lemma \ref{L:paramMI} implies that $h_M$ descends to a smooth function
\begin{equation} \nonumber
  h_M : \sM_I \ \to \ \bR \,.
\end{equation}

\begin{lemma}
Assume that the normalization of Remark \ref{R:normM} is in effect.  The function $h_M$ defines a metric on the line bundle $\Lambda_M \to \Gamma_I\backslash \sM_I$ with curvature form $-\partial\bar\partial \log h_M$, and Chern form
\begin{equation}\label{E:chern}
  c_1(\Lambda_M) \ = \ 
  - \frac{\bi}{2\pi} \partial\bar\partial \log h_M
  \ = \ 
  \half \sQ(M \,,\ [\td x , \td \overline x]) \,.
\end{equation} 
\end{lemma}

\begin{proof}
By \cite[p.~310--311]{MR1288523} it suffices to show
\begin{equation} \label{E:prfhM}
  h_M(\gamma \cdot \zeta(g,x)) \ = \  
  h_M(\zeta(g,x)) \, | e_M(\gamma,\zeta(g,x)) |^{-2}\,.
\end{equation}
Note that \eqref{E:invar} allows us rewrite
\begin{eqnarray}
  \nonumber
  h_M(g,x) & = & 
  \exp \pi \bi 
  \sQ(M\,,\,\tAd_g[ x , \overline x]) \\
  \label{E:hM}
  & = & \exp \pi \bi 
  \sQ(\tAd_g^{-1} M\,,\, [ x , \overline x]) 
  \\  \nonumber 
  & = & \exp \pi \bi 
  \sQ(M\,,\,[ x , \overline x]) \,.
\end{eqnarray}
Let $[x,\overline x]^{-1,-1}$ be the component of $[x,\overline x] \in \fc_{I,\bC}^{-2}$ taking value in $\fc_{I,F}^{-1,-1}$.  It follows from \eqref{E:ds-conj}, \eqref{E:ds-sQ} and \eqref{E:Min} that 
\begin{eqnarray*} 
  h_M(g,x) & = & \exp \pi \bi 
  \sQ\big(M\,,\,\left[x , \overline{x} \right]^{-1,-1} \big) \\ 
  & = & \exp \pi \bi \sum_{p>0} 
  \sQ\big( M \,,\, \big[ x^{-p,p-1} , \overline{x^{-p,p-1}} \big] \big) \,;
\end{eqnarray*}
here $x^{-p,p-1}$ is the component of $x \in \fc_{I,\bC}^{-1}$ taking value in $\fc_{I,F}^{-p,p-1}$, cf.~\eqref{E:ds-cIa}.  Now \eqref{E:y1} yields
\begin{eqnarray*}
  h_M(\gamma \cdot \zeta(g,x)) & = & 
  h_M(\zeta(h,y)) \\
  & = & 
  \exp \pi \bi \sum_{p>0} \sQ
  \big(M\,,\,\big[( x + \tAd_g^{-1} b)^{-p,p-1} , \overline{( x + \tAd_g^{-1} b)^{-p,p-1}} \big] \big) \,.
\end{eqnarray*}
Keeping in mind $x + \tAd_g^{-1}b \in \fc_{I,\bC}^{-1}$ and \eqref{E:cIab}, another application of \eqref{E:lieb}, \eqref{E:ds-sQ} and \eqref{E:ds-cIa} allows us to write this as 
\[
  h_M(\gamma \cdot \zeta(g,x)) \ = \ 
  \exp \pi \bi  
  \sQ\big( M \,,\, \big[ x + \tAd_g^{-1} b , \overline{x + \tAd_g^{-1} b} \big] \big)
  \,.
\]
And \eqref{E:invar} yields
\[
  h_M(\gamma \cdot \zeta(g,x)) \ =  \ 
  \exp \pi \bi  
  \sQ\big( M \,,\, \big[ \tAd_g x + b , \tAd_g \overline{x} +  b \big] \big)
  \,.
\]
Then \eqref{E:eM3} and \eqref{E:hM} yield the desired \eqref{E:prfhM}.
\end{proof}

\subsection{The line bundle $\Lambda_M$ descends to $\Gamma_I\bs \sM_I^1$} \label{S:LMdescends}

\begin{lemma} \label{L:LMdescends}
Assume that the normalization of Remark \ref{R:normM} is in effect.  The line bundle $\Lambda_M \to \Gamma_I \bs \sM_I$ defined in \eqref{E:LM} descends to $\Gamma_I \bs \sM_I^1 = (\Gamma_I \, C_{I,\bC}^{-2}) \bs \sM_I$.
\end{lemma}

\begin{proof}
To prove the lemma it suffices to show that the functions $e_M(\gamma , \cdot) : \sM_I \to \bC^*$ are constant on the fibres of the map $p_1 : \sM_I \to \sM_I^1$ defined in \eqref{E:p1p0}.  It follows from \eqref{SE:gpq} that 
\begin{equation}\nonumber 
  \fg_\bC \ = \ \bigoplus_{a=-\sfn}^\sfn 
  E(a) \,,\quad \hbox{ where } \quad
  E(a) \ = \ \bigoplus_{p+q=a} \fg^{p,q}_{W,F} \,.
\end{equation}
It follows from \eqref{E:ds-cIa} that $\fc_{I,\bC}^{-1} \subset \bigoplus_{a\ge1}\, E(-a)$.  So in order to prove that the function $e_M^\gamma$ is constant on the fibres of $p_1 : \sM_I \to \sM_I^1$ we need to show that $e_M(\gamma,\zeta(g,x)) = \exp \pi \bi \sQ ( M , [b,\tAd_g x])$ is independent of the component of $x \in \bigoplus_{a\ge1}\, E(-a)$ taking value in $\bigoplus_{a\ge2}\, E(-a)$.  Keep in mind that $g \in L_{I,\bR}$ and $b,x,\tAd_gx \in \fc_{I,\bC}^{-1}$.

From \eqref{E:ds-sQ} we see that $\sQ(E(a) , E(b)) \not= 0$ if and only if $a+b=0$.
By \eqref{E:Min}, $M \in E(2)$.  So $\sQ(M,[b,\tAd_gx])$ depends only on the component of $[b,\tAd_gx] \in \fc_{I,\bC}^{-2}$ taking value in $E(-2)$.  By \eqref{E:lieb}, we have $[ E(a) , E(b) ] \subset E(a+b)$.
So $\sQ(M,[b,\tAd_gx])$ depends only on the component of $\tAd_gx$ taking value in $E(-1)$.  From the definition of $Y$ in \S\ref{S:sl2} we see that $E(a)$ is the $a$--eigenspace of $\tad_Y : \fg_\bC \to \fg_\bC$.  By definition \eqref{E:cGIY} of $L_{I,\bR}$ these eigenspaces are preserved by the adjoint action of $L_{I,\bR}$.  So $\sQ(M,[b,\tAd_gx])$ depends only on $g \in L_{I,\bR}$ and the component of $x$ taking value in $E(-1)$.
\end{proof}

\begin{remark}\label{R:descends}
The following corollary of the proof will be useful in \S\ref{S:prf-mark3}.  Fix $\gamma = \a \exp(b)$ and set $g = \mathrm{Id}$.  Regard $\exp\pi\bi \sQ(M,[b,x])$ as a function of $x \in \ff^\perp \cap \fc_{I,\bC}^{-1}$.  The proof of Lemma \ref{L:LMdescends} implies that $\exp\pi\bi \sQ(M,[b,x])$ descends to a well-defined function of $\fc_{I,\bC}^{-2} \bs (\ff^\perp \cap \fc_{I,\bC}^{-1})$.
\end{remark}

\begin{lemma}
The metric $h_M$ on $\Lambda_M \to \Gamma_I\bs \sM_I$ descends to a metric on the line bundle $\Lambda_M \to \Gamma_I \bs \sM_I^1$.
\end{lemma}

\begin{proof}
This follows from \eqref{E:hM} and an argument that is essentially identical to the proof of Lemma \ref{L:LMdescends}.
\end{proof}

\subsection{The fibre bundle structure of $\pi_0 : \Gamma_I \bs \sM_I^1 \to \Gamma_I \bs \sD_I$}

Recall the projections $p : \sM_I \to \sD_I$ and $p_0 : \sM_I^1 \to \sD_I$ of \eqref{E:p} and \eqref{E:p1p0}.  Define
\begin{equation}\label{E:EC}
  E_\bC \ = \ 
  \frac{\fc_{I,\bC}^{-1}}{\fc_{I,\bC}^{-2} + (\fc_{I,\bC}^{-1} \cap \ff)}
  \ \simeq \ 
  \fc_{I,\bC}^{-1} \,\cap\, \ff^\perp
  \,\cap\, E(-1) \ = \ 
  \bigoplus_{p>0} \fc_{I,F}^{-p,p-1} \,.
\end{equation}

\begin{lemma}\label{L:fibre-p0}
Given $\zeta(g,x) = g \exp(x) \cdot F \in \sM_I$, the $p_0$--fibre over $p(g \cdot F) = g \cdot p(F) \in \sD_I$ is biholomorphic to $\tAd_g(E_\bC)$.
\end{lemma}

\begin{proof}
It follows from \eqref{E:cIab} and Remark \ref{R:exp(cIa)} that the exponential map induces a natural identification $C_I^{-2} \bs C_I^{-1} \simeq \fc_I^{-2} \bs \fc_I^{-1}$.  Lemma \ref{L:fibre-p0} now follows from Lemma \ref{L:fibre-p} and the definition \eqref{E:p1p0} of $p_0$.
\end{proof}

Let $E_\bZ$ denote the image of $\Gamma_I^{-1}$ under the projection $C_{I,\bC}^{-2} \bs C_{I,\bC}^{-1} \to C_{I,\bC}^{-2} \bs (C_{I,\bC}^{-1} \cdot F) \simeq E_\bC$.  Then \eqref{E:ds-conj} implies that $E_\bZ \simeq \Gamma_I^{-2} \bs \Gamma_I^{-1}$ is a lattice in $E_\bC$.  In particular, $E_\bZ \bs E_\bC$ is a compact complex torus.

Let $\overline p(g \cdot F)$ be the image of $p(g \cdot F) \in \sD_I$ under the projection $\sD_I \to \Gamma_I \bs \sD_I$.  Recall the map $\pi_0 : \Gamma_I \bs \sM_I^1 \to \Gamma_I \bs \sD_I$ of \eqref{E:towerIwt}.

\begin{corollary} \label{C:pi0-fibre}
The $\pi_0$--fibre over $\overline p(g \cdot F) \in \Gamma_I \bs \sD_I$ is a finite quotient of the complex torus $\tAd_g(E_\bZ \bs E_\bC)$.
\end{corollary}

\begin{proof}
The $\pi_0$--fibre over $\overline p(g \cdot F)$ is the quotient of the $p_0$--fire $\tAd_g(E_\bC)$ (Lemma \ref{L:fibre-p0}) by the action of $\tStab_{\Gamma_I}(p(g \cdot F))$.  By Lemma \ref{L:cpt-stab-CI} and Remark \ref{R:stab-p(F)}, this stabilizer is $\tStab_{\Gamma_I}(F) \cdot \Gamma_I^{-1}$.  The quotient of the $p_0$--fibre $\tAd_g(E_\bC)$ by $\Gamma_I^{-1}$ is the complex torus $\tAd_g(E_\bZ \bs E_\bC)$.  And $\tStab_{\Gamma_I}(F)$ is finite by Corollary \ref{C:stab-GammaI}.  
\end{proof}

\begin{remark}  
If $\Gamma$ is neat, then the fibres of $\pi_0 : \Gamma_I \bs \sM_I^1 \to \Gamma_I\bs \sD_I$ are compact, complex tori (Corollary \ref{C:pi0-fibre}).  These tori can be interpreted as parameterizing extension data in the mixed Hodge structure $(W,F)$.  To see this, let $\Gamma_W = \Gamma \cap P_{W,\bQ}$ be the subgroup preserving the weight filtration.  Each $F \in \sM_I$ defines a Hodge structure $H_\ell$ on $\tGr^W_\ell$ of weight $\ell$.  Let 
\[
  \mathrm{Ext}^1_\mathrm{MHS}(H^\ell,H^{\ell-1})
  \ = \ 
  \frac{\tHom(H_\ell,H_{\ell-1})}
  {F^0\tHom(H_\ell,H_{\ell-1}) \,+\, \tHom_\bZ(H_\ell,H_{\ell-1})} 
\]
be the group of extensions.  Then $\bigoplus_{\ell=1}^{2\sfn} \mathrm{Ext}^1_\mathrm{MHS}(H_\ell,H_{\ell-1})$ may be identified with the fibre of $(\Gamma_W \, P_{W,\bC}^{-2})\bs(P_{W,\bC} \cdot F) \to (\Gamma_W \, P_{W,\bC}^{-1})\bs(P_{W,\bC} \cdot F)$ over the point $\overline p(F) \in (\Gamma_W \, P_{W,\bC}^{-1})\bs(P_{W,\bC} \cdot F)$ determined by $F$.   If, in a slight abuse of notation, we also let $\overline p(F)$ denote the corresponding point in $\Gamma_I \bs \sD_I$, then the $\pi_0$--fibre over $\overline p(F)$ is a subset of $\bigoplus_{\ell=1}^{2\sfn} \mathrm{Ext}^1_\mathrm{MHS}(H_\ell,H_{\ell-1})$.  If the period domain $\sD$ is hermitian, equality holds.  In general, the containment is strict.
\end{remark}

\subsection{The image of $\Theta_I : \Zw_I \to \Gamma_I \bs \sM_I^1$} \label{S:imageThetaI}

Let $E'_\bC \subset E_\bC$ be the image of $\fc_{I,F}^{-1,0} \subset \fc_{I,\bC}^{-1}$ under the projection $\fc_{I,\bC}^{-1} \to E_\bC$, cf.~\eqref{E:EC}.  The image of $E'_\bC$ in $E_\bZ \bs E_\bC$ is of the form 
\[
  \tim\{ E'_\bC \,\to\, E_\bZ \bs E_\bC \} 
  \ \simeq \ 
  \bC^a \times (\bC^*)^b \times \sJ 
  \ \subset \ E_\bZ \bs E_\bC \,,
\]
with $(\bC^*)^b \times \sJ$ a complex torus having compact factor $\sJ$.

\begin{lemma} \label{L:sJ}
Assume that the normalization of Remark \ref{R:normM} is in effect.  The subtorus $\sJ \subset E_\bZ \bs E_\bC$ is an abelian variety that is polarized by the line bundles $\Lambda_M \to \Gamma_I \bs \sM_I^1$.
\end{lemma}

\begin{proof}[Proof of Lemma \ref{L:sJ}]
It suffices to show that the Chern form $c_1(\Lambda_M)$ of \eqref{E:chern} is positive on $E'_\bC$.  This is a consequence of the Hodge--Riemann bilinear relations for the induced limiting mixed Hodge structure on $(\fg,\sQ)$, cf.~\S\ref{S:induced}.  Given $N \in \s_I$, let 
\[
  P_1(\tad_N,\fg) \ = \ 
  \tker\{ \tad_N^2 : \tGr^W_1(\fg) \to \tGr^W_{-3}(\fg) \} \,.
\]
Let $P_1(\tad_N,\fg_\bC) = \bigoplus_{p+q=1} P_1(\tad_N,\fg)^{p,q}$ be the Hodge decomposition induced by $F(\fg)$.  This Hodge structure is polarized by $\sQ(\cdot , \tad_N \cdot)$.  In particular, 
\begin{equation}\label{E:hpol}
  -\bi\,\sQ(u , \tad_N \overline u) \ > \ 0 \,,\quad \hbox{ for all } \quad 0 \not= u \in P_1(\tad_N,\fg_\bC)^{0,1} \,.
\end{equation}
Let $\fc_N \supset \fc_I$ be the centralizer of $N \in \s_I$, and $\tGr^W_{-1}(\fc_N) = \fc_N^{-2} \bs \fc_N^{-1} \inj \tGr^W_{-1}(\fg)$.  The classical theory of $\fsl_2$--representations implies that the triple $\{M,Y,N\}$ of \S\ref{S:sl2} satisfies
\begin{eqnarray*}
  \tGr^W_{-1}(\fc_N) 
  & = & 
  \tker\{ \tad_N : \tGr^W_{-1}(\fg) \to \tGr^W_{-3}(\fg) \} 
  \ = \ 
  \tad_N( P_1(\tad_N,\fg) ) \,,\\
  P_1(\tad_N,\fg) & = & 
  \tad_M(\tGr^W_{-1}(\fc_N) ) \,.
\end{eqnarray*}
Also, both
\[
  \tad_M \circ \tad_N : P_1(\tad_N,\fg) \ \to \ 
  \tGr^W_{-1}(\fc_N) 
  \tand 
  \tad_N \circ \tad_M : \tGr^W_{-1}(\fc_N) \ \to \ P_1(\tad_N,\fg)
\]
are the identity map.  So for each $0\not=v \in E_\bC' \simeq \fc_{I,\bC}^{-1,0} \subset \fc_{N,\bC}^{-1,0}$ there is a unique $0\not=u = \tad_M v \in P_1(\tad_N,\fg_\bC)^{0,1}$ so that $\tad_N(u) = v$.  Then
\begin{eqnarray*}
  -\bi c_1(\Lambda_M)(v,\overline v) & = & 
  -\half \bi\, \sQ( M \,,\, [v , \overline v])
  \ = \ 
  \half \bi\, \sQ( v \,,\, \tad_M(\overline v)) 
  \\ & = & 
  \half \bi\, \sQ( \tad_Nu \,,\, \overline u)
  \ = \ 
  -\half \bi \, \sQ( u \,,\, \tad_N\overline u)
  \,,
\end{eqnarray*}
and the lemma follows from \eqref{E:hpol}.
\end{proof}

Since the metric $h_M(g,x)$ on $\Lambda_M$ does not depend on $g$, cf.~\eqref{E:hM}, the proof of the lemma shows that the subtorus $\tAd_g\sJ \subset \tAd_g(E_\bZ \bs E_\bC)$ is an abelian variety that is polarized by $\Lambda_M$.  Recall the maps
\[
  \begin{tikzcd}
    \Zw_I \arrow[r,"\Theta_I"]
    	\arrow[rd,"\Phi_I"',end anchor={north west}]
    & \Gamma_I \bs \sM_I^1 \arrow[d,"\pi_0"] \\
    & \ \Gamma_I \bs \sD_I \,.
  \end{tikzcd}
\]
of \eqref{E:towerIwt}. Let $A \subset \Zw_I$ be a connected component of a $\Phi_I$--fibre.  The image $\Theta_I(A)$ is contained in a $\pi_0$--fibre.  By Corollary \ref{C:pi0-fibre} the $\pi_0$--fibre is a finite quotient of the complex torus $\tAd_g(E_\bZ \bs E_\bC)$. 

\begin{lemma} \label{L:ThetaI(A)}
Let $A \subset \Zw_I$ be a connected component of a $\Phi_I$--fibre.  The image $\Theta_I(A)$ \emph{(which is contained in a $\pi_0$--fibre)} is contained in a (finite quotient of a) translate $a+\tAd_g(\sJ)$ 
\end{lemma}

\begin{proof}[Proof of Lemma \ref{L:ThetaI(A)}]
This is a consequence of the infinitesimal period relation (\S\ref{S:ipr}).  Locally about $o \in Z_J^* \subset \Zw_I$ we have $(g{}^{-1} \td g)^{p,q}=0$ for all $p+q \ge 1$, by Remark \ref{R:ipr-tg*} and keeping in mind that $W(\s_J) = W(\s_I)$.  The infinitesimal variation of $\Theta_I$ (resp.~$\Phi_I$) along $\Zw_I$ is encoded by the $(g{}^{-1} \td g)^{p,q}$ with $0 \le p + q \le -1$ (resp.~$0 = p + q$).  So the infinitesimal variation in $\Theta_I$ along $A$ is encoded by the $(g{}^{-1} \td g)^{p,q}$ with $p+q=-1$.  The infinitesimal period relation \eqref{E:ipr-tg} implies $\left.(g{}^{-1} \td g)\right|_A$ takes value in $\fc_{I,\bC}^{-1,0}$.  Since $A$ is connected, this implies that $\Theta_I(A)$ lies in a translate of $\bC^a \times (\bC^*)^b \times \sJ \simeq E_\bZ \bs E'_\bC \inj E_\bZ \bs E_\bC$.  And since $A$ is compact, and $\Theta_I : A \to \bC^a \times (\bC^*)^b \times \sJ$ is holomorphic, the image must lie in a translate of $\sJ$.
\end{proof}

\section{Theta bundles versus normal bundles}

Recall the line bundles $\Lambda_M$ over $\Gamma_I \bs \sM_I^1$ polarizing the abelian varieties $\sJ$ of
\[
\begin{tikzcd}[column sep=small]
  \sJ \arrow[r,hook] 
  & E_\bZ \bs E_\bC \arrow[r,hook]
  & \Gamma_I \bs \sM_I^1 \arrow[d,"\pi_0"] \\
  & & \Gamma_I\bs \sD_I \,,
\end{tikzcd}
\]
cf.~\S\ref{S:LMdescends} and \S\ref{S:imageThetaI}.  The main result of this section establishes a relationship between the line bundles $\Lambda_M$, and the normal bundles $[Z_i] = \cN_{Z_i/\olB}$.  Recall the maps
\[
  \begin{tikzcd}
    \Zw_I \arrow[r,"\Theta_I"]
    	\arrow[rd,"\Phi_I"',end anchor={north west}]
    & \Gamma_I \bs \sM_I^1 \arrow[d,"\pi_0"] \\
    & \ \Gamma_I \bs \sD_I \,.
  \end{tikzcd}
\]
of \eqref{E:towerIwt}.

\begin{theorem}\label{T:mark}
Assume $M$ is normalized as in Remark \ref{R:normM}.  Let $A \subset \Zw_I$ be a connected component of a $\Phi_I$--fibre.  We have
\begin{equation}\label{E:mark}
  \left.\Theta_I^*(\Lambda_M)\right|_A
  \ = \ 
  \sum_j \left.\sQ(M,N_j) [Z_j]\right|_A \,,
\end{equation}
with $\sQ(M,N_i) \in \bZ$, and summing over all $Z_j \cap A \not= \emptyset$.
\end{theorem}

\noindent The assertion $\sQ(M,N_i) \in \bZ$ follows directly from the normalization of Remark \ref{R:normM}.  The theorem is proved in \S\S\ref{S:prf-mark1}--\ref{S:prf-mark3}. 

\begin{definition}
Given an $\fsl_2$--triple $\{ M , Y , N\}$, normalized as in Remark \ref{R:normM}, we say that $M$ \emph{is integral with respect to the cone} $\s_I$ if $N \in \s_I$, and $0 < \sQ(M,N_i) \in \bZ$ for every generator $N_i$ of $\s_I$.
\end{definition}

\noindent From Theorem \ref{T:imageThetaI(A)}, Lemmas \ref{L:sJ} and \ref{L:ThetaI(A)}, and Theorem \ref{T:mark} we deduce

\begin{corollary} \label{C:mark}
Let $A \subset \Zw_I$ be a connected component of a $\Phi_I$--fibre.  Assume that $M$ is integral with respect to $\s_I$, and that the differential of $\left.\Theta_I\right|_A$ is injective.  Then the line bundle $-\sum\left.\sQ(M,N_j)\,\cN^*_{Z_j/\olB}\right|_A$ is ample. 
\end{corollary}

\subsection{Proof of Theorem \ref{T:mark}.  Step 1: monodromy near a $\Phi_I$--fibre} \label{S:prf-mark1}

We begin with the following lemma.  Fix an element $\{T_i\}_{i\in I} \subset \tGL(V_\bZ)$ in the $\Gamma$--conjugacy class $\cT_I$ defined in Remark \ref{R:cTI}.  This choice determines nilpotent operators $N_i = \log T_i \in \fgl(V_\bQ)$, which in term determine the weight filtration $W$ (\S\ref{S:W}), and centralizers $C_I$ and $\Gamma_I$ (\S\ref{S:MI} and \S\ref{S:PsiI}).

\begin{lemma}[{\cite{GGR-part1}}] \label{L:nbd-PhiIfibre}
Let $A \subset \Zw_I$ be a connected component of a $\Phi_I$--fibre.  Let $(W,F)$ be any mixed Hodge structure arising along $A$, as in \S\S\ref{S:vhs}--\ref{S:LMHS}.  Let $\sS = \exp(\ff^\perp) \cdot F \subset \check \sD$ be the associated Schubert variety, cf.~\S\ref{S:schubert}.  Let 
\[
  \Gamma_{I,\infty} \ = \ 
  \tStab_{\Gamma_I}(F_\infty)
\]
be the stabilizer in $\Gamma_I$ of the filtration $F_\infty$ defined in \S\ref{S:schubert}.  
\begin{i_list_emph}
\item \label{i:Finfty}
The filtration $F_\infty$ is independent of our choice of $(W,F)$ along $A$.  
\item \label{i:schubert}
The action of $\Gamma_{I,\infty}$ on $\check \sD$ preserves $\sS$.  
\item \label{i:lift-infty}
There exists a neighborhood $X \subset \olB$ of $A$ so that the restriction of the period map $\Phi$ to $U = B \cap X$ lifts to $\Gamma_{I,\infty}\bs (\sD \cap \sS)$: there is a commutative diagram
\[
  \begin{tikzcd}
  & \Gamma_{I,\infty} \bs (\sD \cap \sS) \arrow[d] \\
  U \arrow[r,"\Phi"'] \arrow[ru,"\Phi_{I,\infty}"]
  & \ \Gamma \bs \sD \,.
  \end{tikzcd}
\]
The restrictions of $\Phi_I$ and $\Theta_I$ to $\Zw_I \cap X$ are both proper.
\end{i_list_emph}
\end{lemma}

\begin{proof}
Part \emph{\ref{i:Finfty}} is \cite[Proposition 4.2]{GGR-part1}.  Part \emph{\ref{i:schubert}} is a direct consequence of the definition \eqref{E:Sp} of $\sS$.  With the exception of the final assertion on the properness of $\Theta_I$, part \emph{\ref{i:lift-infty}} follows from \cite[Lemma 4.15 and Remark 4.17]{GGR-part1}.  The properness of $\Theta_I$ follows from that of $\Phi_I$, as in the proof of Lemma \ref{L:proper2}.
\end{proof}

\subsection{Proof of Theorem \ref{T:mark}.  Step 2: a line bundle $\tilde \Lambda_M$ over $\Gamma_{I,\infty}\bs (\sD \cap \sS)$} \label{S:prf-mark2}

Let $F \in \sM_I$ be as in Lemma \ref{L:nbd-PhiIfibre}.  Given $N \in \s_I$, let $\{ M , Y , N\}$ be the associated $\fsl_2$--triple of \S\ref{S:sl2}.  By Lemma \ref{L:nbd-PhiIfibre}\emph{\ref{i:schubert}}, we have a well-defined action
\[
  \Gamma_{I,\infty} \,\times\, \sS
  \ \to \ \sS
\]
of $\Gamma_{I,\infty}$ on the Schubert cell $\sS = \exp(\ff^\perp) \cdot F$.  Recall the biholomorphism $\lambda : \sS \to \ff^\perp$ of \S\ref{S:pmr}: if $s \in \sS$, then $s = \exp(\lambda(s)) \cdot F$.
Define
\[
  \tilde f_M : \Gamma_{I,\infty} \,\times\,
  \sS \ \to \ \bC^*
\]
by
\[
  \tilde f_M(\gamma,s) \ = \ 
  \exp
  2 \pi \bi\,\sQ(M , \lambda(\gamma \cdot s)) \,.
\]
Define
\[
  \tilde e_M : \Gamma_{I,\infty} \,\times\,
  \sS \ \to \ \bC^* 
\]
by
\[
  \tilde e_M(\gamma,s) \ = \ 
  \frac{\tilde f_M(\gamma,s)}{\tilde f_M(1,s)}
  \,.
\]
Then $\tilde e_M(\gamma_2\gamma_1,s) = \tilde e_M(\gamma_2 , \gamma_1 \cdot s) \, \tilde e_M(\gamma_1,s)$.  (Note that $\tilde f_M(\gamma_2\gamma_1,s) = \tilde f_M(\gamma_2 , \gamma_1 \cdot s)$.)  So $\tilde e_M$ is a factor of automorphy defining a line bundle $\tilde \Lambda_M$ over $\Gamma_{I,\infty} \bs (\sD \cap \sS)$, and $\tilde f_M$ defines a section of this line bundle.

\begin{lemma} \label{L:mark}
Recall the map $\Phi_{I,\infty} : U \to \Gamma_{I,\infty} \bs (\sD \cap \sS)$ of Lemma \ref{L:nbd-PhiIfibre}.  We have 
\begin{equation} \label{E:prf-mark1}
    \Phi_{I,\infty}^*(\tilde\Lambda_M)
    \ = \ 
    \sum_j \sQ(M,N_j) [Z_j \cap X] \,. 
\end{equation}
The sum is over all $Z_j \cap A \not= \emptyset$.
\end{lemma}

\begin{proof}
Fix a local lift $\tPhi(t,w)$ of $\Phi_{I,\infty}$, on a coordinate chart $\olU_o$ centered at a point $o \in A$, as in \S\ref{S:loclift}.  Lemma \ref{L:nbd-PhiIfibre} implies there is a holomorphic function $g_o(t,w)$ taking value in $\exp(\ff^\perp)$ so that $\tPhi(t,w) = \exp( \sum \ell(t_j) N_j) g_o(t,w) \cdot F$.  Regard $\tPhi(t,w)$ as a multi-valued (due to the logarithms $\ell(t_j)$) function taking value in $\sS$.  Then \eqref{E:ds-sQ}, \eqref{E:Min} and \S\ref{S:pmr}\ref{i:11} imply
\begin{equation}\label{E:tfM}
  \tilde f_M(1,\tPhi(t,w)) \ = \ 
  \exp 2 \pi \bi \sQ(M,\log g_o(t,w))
  \, \prod_j t_j^{\sQ(M,N_j)}\,.
\end{equation}
Here the product is over all $j \in J$, with $J$ determined by $o \in Z_J^*$.  Any other local lift at $o$ is of the form 
\[
  \gamma \cdot \tPhi(t,w) 
  \ = \ 
  \exp(\sum \ell(t_j) N_j) \gamma g_o(t,w) \cdot F
  \ = \ 
  \exp(\sum \ell(t_j) N_j) g_o^\gamma(t,w) \cdot F
\]
for some $\gamma \in \Gamma_{I,\infty}$, and the holomorphic $g^\gamma_o : \olU_o \to \exp(\ff^\perp)$ determined by $\gamma g_o(t,w) \cdot F = g_o^\gamma(t,w) \cdot F$.  So 
\[
  \tilde f_M( \gamma , \tPhi(t,w))
  \ = \ 
  \tilde f_M( 1 , \gamma \cdot \tPhi(t,w))
  \ = \ 
  \exp 2 \pi \bi \sQ(M,\log g^\gamma_o(t,w))
  \, \prod_j t_j^{\sQ(M,N_j)}\,.
\]
In particular, while $\tPhi(t,w)$ is defined only on $\sU_o = B \cap \olU_o$, the $\tilde f_M(\gamma,\tPhi(t,w))$ extend to meromorphic functions on all of $\olU_o$, cf.~\eqref{SE:g-FW}, \eqref{SE:gpq} and \eqref{E:tilde-g}. 

\begin{remark} \label{R:eM}
Likewise, the 
\[
  e_M(\gamma,\tPhi(t,w)) \ = \ 
  \frac{\exp 2 \pi \bi \sQ(M,\log g^\gamma_o(t,w))}{\exp 2 \pi \bi \sQ(M,\log g_o(t,w))}
\]
extend to nowhere vanishing holomorphic functions on all of $\olU_o$.
\end{remark}

Let $U = B \cap X$ and $\Phi_{I,\infty} : U \to \Gamma_{I,\infty} \bs (\sD \cap \sS)$ be as in Lemma \ref{L:nbd-PhiIfibre}.  The map $\Phi_{I,\infty}$ lifts to the universal covers
\[
\begin{tikzcd}
  \widetilde U \arrow[d] \arrow[r,"\tPhi_{I,\infty}"]
  & \sD \cap \sS \arrow[d] \\
  U \arrow[r,"\Phi_{I,\infty}"] 
  & \Gamma_{I,\infty} \bs (\sD \cap \sS) \,.
\end{tikzcd}
\]
Let $\rho : \pi_1(U) \to \Gamma_{I,\infty}$ be the monodromy representation.  Then the local coordinate computations above imply that $e_M$ pulls back under $(\rho,\tPhi_{I,\infty})$ to a factor of automorphy 
\[
  (\rho,\tPhi_{I,\infty})^*(e_M) : \pi_1(U) \times \widetilde U \ \to \ \bC^* \,.  
\]
This factor of autormphy defines the line bundle $\Phi_{I,\infty}^*(\tilde\Lambda_M)$ over $U$.  Remark \ref{R:eM} implies that $\Phi_{I,\infty}^*(\tilde\Lambda_M)$ extends to all of $X$.  The pullback $(\rho,\tPhi_{I,\infty})^*f_M$ defines a section of $\Phi_{I,\infty}^*(\tilde\Lambda_M)$.  And from \eqref{E:tfM} we see that \eqref{E:prf-mark1} holds.
\end{proof}

\subsection{Proof of Theorem \ref{T:mark}.  Step 3: comparison of line bundles $\Phi_{I,\infty}^*(\tilde \Lambda_M)$ and $\Theta_I^*(\Lambda_M)$} \label{S:prf-mark3}

It follows from Lemma \ref{L:mark} that in order to prove \eqref{E:mark} it suffices to show that the factors of automorphy $(\rho,\Theta_I)^*(e_M)$ and $(\rho,\tPhi_I)^*(\tilde e_M)$ defining $\Theta_I^*(\Lambda_M)$ amd $\Phi_I^*(\tilde\Lambda_M)$, respectively, coincide on $A$.  An important subtlety to keep in mind is that $e_M$ is defined relative to Hodge filtration $F$ with the property that the semisimple operator $Y = Y(W,F)$ is rational (\S\ref{S:GammaIact}), while $\tilde e_M$ is defined with respect to a Hodge filtration arising along $A$.  We will continue to denote the latter by $F$: let $F$ be as in Lemma \ref{L:nbd-PhiIfibre} and \S\ref{S:prf-mark2}.  There exists $\hat F \in C_{I,\bC}^{-1}\cdot F \subset \sM_I \cap \sS$ so that the semisimple operator $Y = Y(W,\hat F)$ is rational (Remark \ref{R:samefibre}).  Without loss of generality we may assume that the Hodge filtration of \S\ref{S:GammaIact} is this $\hat F$.  Then we have biholomorphisms (Remark \ref{R:CInegF}) 
\begin{equation}\label{E:vhat}
  \hat\ff{}^\perp \,\cap\, \fc_{I,\bC}^{-1} 
  \ \simeq \ 
  C_{I,\bC}^{-1} \cdot \hat F
  \ = \ C_{I,\bC}^{-1} \cdot F
  \ \simeq \ 
  \ff^\perp \,\cap\, \fc_{I,\bC}^{-1} \,.
\end{equation}
Since $\Gamma_{I,\infty}$ stabilizes $F_\infty$ (by definition), we see from \eqref{E:eqstab} that the action of $\Gamma_{I,\infty}$ on $\check \sD$ preserves the $p$--fibre $C_{I,\bC}^{-1} \cdot F$.  The biholomorphisms \eqref{E:vhat} define an induced action of $\Gamma_{I,\infty}$ on both $\hat\ff{}^\perp \cap \fc_{I,\bC}^{-1}$ and $\ff^\perp \cap \fc_{I,\bC}^{-1}$.  By construction the biholomorphisms \eqref{E:vhat} are $\Gamma_{I,\infty}$--equivariant.

Since $A$ is contained in a $\Phi_I$--fibre, at any point $(0,w) \in A \cap \olU_o$ we have $g_o(0,w) \cdot F \in C_{I,\infty}^{-1} \cdot F$ and $\tilde x(w) = \log g_o(0,w) \in \ff{}^\perp \,\cap\, \fc_{I,\bC}^{-1}$.  In particular, 
\[
  (\rho,\tPhi)^*\tilde e_M(\gamma;0,w) \ = \
  e_M(\gamma,\tPhi(0,w)) \ = \ 
  \frac{\exp 2 \pi \bi \sQ(M,\gamma \cdot \tilde x(w))}
  {\exp 2 \pi \bi \sQ(M,\tilde x(w))}
\]
(Remark \ref{R:eM}).  A computation that is essentially equivalently to the derivation of \eqref{E:eM2} and \eqref{E:eM3} yields
\begin{equation}\label{E:pbteM}
  (\rho,\tPhi)^*\tilde e_M(\gamma;0,w) \ = \
  \exp \bi \pi \sQ(M \,,\, [b , \tilde x(w)] )\,,
\end{equation}
with $\gamma = \a \exp(b)$ the decomposition of $\gamma$ with respect to $C_I = L_I \ltimes C_I^{-1}$, cf.~\S\ref{S:GammaIact}.  On the other hand, if $x(w) \in \hat \ff{}^\perp \cap \fc_{I,\bC}^{-1}$ is the image of $\tilde x(w)$ under the biholomorphism \eqref{E:vhat}, then 
\begin{equation}\label{E:xvtx}
  g_o(0,w) \cdot F \ = \ 
  \exp \tilde x(w) \cdot F \ = \ 
  \exp x(w) \cdot \hat F \,,
\end{equation}
and \eqref{E:eM3} yields
\begin{equation} \label{E:pbeM}
  (\rho,\tPhi)^* e_M(\gamma;0,w) \ = \ 
  \exp \bi \pi \sQ(M \,,\, [b , x(w)] ) \,.
\end{equation}
To see that \eqref{E:pbteM} and \eqref{E:pbeM} are equal, it suffices to show that 
\begin{equation}\label{E:pbeq}
  \tilde x(w) \ \equiv \ x(w)
  \quad\hbox{modulo}\quad \fc_{I,\bC}^{-2}
\end{equation}
(Remark \ref{R:descends}).  Since $F$ and $\hat F$ lie in the same $p$--fibre, they induce the same Hodge structure on the quotient $\fc^{-2}_I \bs \fc^{-1}_I$; that is, $\fc_{I,F}^{p,q} \equiv \fc_{I,\hat F}^{p,q}$ modulo $\fc_{I,\bC}^{-2}$ for all $p+q=-1$.  The desired \eqref{E:pbeq} then follows from the second equality of \eqref{E:xvtx}.

This completes the proof of \eqref{E:mark}.  It remains to show that we can choose $M$ so that the coefficients $\sQ(M,N_j)$ are negative integers.

\section{The relationship between $\Psi_I$ and $\Theta_I$} \label{S:Psi-v-Theta}

Recall the maps 
\[
  \begin{tikzcd}
    \Zc_I \arrow[d,hook] \arrow[r,"\Psi_I"]
    & (\Gamma_I \exp(\bC\s_I)) \bs \sM_I
    \arrow[d,"\pi_1"]
    \\
    \Zw_I
    \arrow[r,"\Theta_I"]
    \arrow[rd,"\Phi_I"']
    & \Gamma_I \bs \sM_I^1 \arrow[d,"\pi_0"] \\
    & \, \Gamma_I \bs \sD_I \,.
  \end{tikzcd}
\]
of Lemma \ref{L:extn2Zc} and \eqref{E:towerIwt}.  If it is the case that $\Zc_I = \Zw_I$, then $\Psi_I$ is locally constant on the fibres of $\Theta_I$ (Corollary \ref{C:thatsall}).  That is, up to what are essentially constants of integration, $\Psi_I$ is determined by $\Theta_I$.  In general the containment $\Zc_I \subset \Zw_I$ may be strict.  The main result of this section is that modulo a certain quotient (constructed from the monodromy cones $\s_J$ along the $\Zc_J \subset \Zw_I \cap \overline{\Zc_I}$, \S\ref{S:sX}), the map $\Psi_I$ is locally constant on the fibres of $\Theta_I$ (Theorem \ref{T:thatsall}).  Moreover, the information in the quotient is not lost: it is encoded in sections of line bundles $\tilde\Lambda_{M'}$ (Remark \ref{R:thatsall}).

\subsection{Monodromy near a $\Theta_I$--fibre}

Let $A' \subset \Zw_I$ be a connected component of a $\Theta_I$--fibre with the property that $\Zc_I \cap A' \not= \emptyset$.  Fix a mixed Hodge structure $(W,F)$ arising along $\Zc_I \cap A'$, as in \S\S\ref{S:vhs}--\ref{S:LMHS}.  Note that $A'$ is necessarily contained in a connected component $A \subset \Zw_I$ of a $\Phi_I$--fibre; we assume the notations of Lemma \ref{L:nbd-PhiIfibre}.

\begin{lemma}
Assume that $\Gamma$ is neat.   Then $\Gamma_{I,\infty} \subset C_{I,\bQ}^{-1}$ is unipotent.
\end{lemma}

\begin{proof}
From Corollary \ref{C:stab-GammaI} and \eqref{E:eqstab} we see that $\varrho(\Gamma_{I,\infty})$ is finite.  Since $\Gamma$ is neat, $\varrho(\Gamma_{I,\infty})$ must be trivial; equivalently, $\Gamma_{I,\infty} \subset C_{I,\bQ}^{-1}$.
\end{proof}

Define 
\[
  \Gamma_{I,\infty}^{-2} \ = \ \Gamma_{I,\infty} \,\cap \,C_{I,\bQ}^{-2} \,.
\]
The following lemma (which is an analog of Lemma \ref{L:nbd-PhiIfibre}\eref{i:lift-infty}) describes the monodromy near $A'$.

\begin{lemma} \label{L:Phi-2}
Let $A' \subset \Zw_I$ be a connected component of a $\Theta_I$--fibre.  There exists a neighborhood $X \subset \olB$ so that the restriction of the period map $\Phi$ to $U = B \cap X$ lifts to $\Gamma_{I,\infty}^{-2}\bs (\sD \cap \sS)$: there is a commutative diagram
\[
  \begin{tikzcd}
  & \Gamma_{I,\infty}^{-2} \bs (\sD \cap \sS) \arrow[d] \\
  U \arrow[r,"\Phi"'] \arrow[ru,"\Phi_{I,\infty}^{-2}"]
  & \ \Gamma \bs \sD \,.
  \end{tikzcd}
\]
\end{lemma}

\begin{proof}
The argument is a slight refinement of the proof of \cite[Lemma 4.15]{GGR-part1}.  Given $o \in A'$, fix a coordinate neighborhood $\olU_o \subset \olB$ centered at $o$, and a local lift $\tPhi_o : \widetilde\sU_o \to \sD$ of $\left.\Phi\right|_{\sU_o}$, as in \S\ref{S:vhs}.  Then $\tPhi_o(t,w) = \exp(\sum \ell(t_j) y_j) g_o(t,w) \cdot F_o$, with $g_o$ satisfying \eqref{E:tilde-g2} and $g_o(0,0)=0$ (Remark \ref{R:tg=0}).  Let $F_\infty(o) = \lim_{y\to\infty} \exp(yN) \cdot F_o$ be the associated limit filtration of \eqref{E:Finfty}.  By Lemma \ref{L:nbd-PhiIfibre}\eref{i:Finfty}, we may choose the lifts $\tPhi_o$ so that the $F_\infty(o)$ are independent of $o \in A'$.  And because $A'$ is a connected component of a $\Theta_I$--fibre, we may refine this choice of lift so that $(F^p_o \cap W_\ell)/(F^p_o \cap W_{\ell-2})$ is independent of $o$ for all $p,\ell$.  This determines the lift $\tPhi_o$ up to the action of $\Gamma_{I,\infty}^{-2}$.  Define $X = \cup_{o \in A'} \olU_o$.  Since the local lifts are defined up to the action of $\Gamma_{I,\infty}^{-2}$, we can patch the local lifts $\{ \tPhi_o : \widetilde\sU_o \to \sD\}_{o \in A'}$ together to define the map $\Phi_{I,\infty}^{-2} : U = B \cap X \to \Gamma_{I,\infty}^{-2} \bs (\sD \cap \sS)$.
\end{proof}

\begin{remark} \label{R:FA'}
At the beginning of this section, we fixed a mixed Hodge structure $(W,F)$ arising along $\Zc_I \cap A'$.  Without loss of generality, this mixed Hodge structure is one of the $(W,F_o)$ in the proof of Lemma \ref{L:Phi-2}.  Then, for any $o \in A'$, we have $(F^p_o \cap W_\ell)/(F^p_o \cap W_{\ell-2}) = (F^p \cap W_\ell)/(F^p \cap W_{\ell-2})$ for all $p,\ell$.  In particular, $F_o \in C_{I,\bC}^{-2} \cdot F \subset \sM_I$ for all $o \in A'$.  
\end{remark}

\subsubsection{Local coordinate representations} \label{S:coord-rep}

The normalization of Remark \ref{R:FA'} allows us to re-express the local lifts in the proof of Lemma \ref{L:Phi-2} as $\tPhi_o(t,w) = \exp (\sum \ell(t_j) N_j) \hat g_o(t,w) \cdot F$, with $\hat g_o(t,w) \cdot F = g_o(t,w) \cdot F_o$ and $\hat g_o: \olU_o \to \exp(\ff^\perp)$ holomorphic.  It will be helpful to note that: 
\begin{a_list}
\item 
We have $\log \hat g_o(0,w) \in \ff^\perp \cap \fc_{I,\bC}$ for all $(0,w) \in \olU_o$.
\item
The local coordinate representation of $\Psi_I$ is (\S\ref{S:pmr})
\[
  \Psi_I(0,w) \ \equiv \ 
  \log \hat g_o(0,w)
  \ (\hbox{mod } \bC\s_I)
  \ \in \ \frac{\ff^\perp \cap \fc_{I,\bC}}{\bC\s_I} \,.
\]
\item
Likewise, the local coordinate representation of $\Theta_I$ is 
\[
  \Theta_I(0,w) \ \equiv \ 
  \log \hat g_o(0,w)
  \ (\hbox{mod } \ff^\perp \cap \fc_{I,\bC}^{-2})
  \ \in \ \frac{\ff^\perp \cap \fc_{I,\bC}}{\ff^\perp \cap \fc_{I,\bC}^{-2}} \,.
\]
\item \label{i:dg-2}
We have $\log \hat g_o(0,w) \in \ff^\perp \cap \fc_{I,\bC}^{-2}$ for all $(0,w) \in A'$.  In particular, 
\[
  \left.\left(\td \log \hat g_o(0,w)\right)^{p,q}
  \right|_{A' \cap \olU_o} \ = \ 0 \,,\quad
  \forall \quad p+q \ge -1 \,.
\]
Here, $\left(\td \log \hat g_o(0,w)\right)^{p,q}$ is the component of $\td \log \hat g_o(0,w)$ taking value in $\fc_{I,F}^{p,q}$.
\end{a_list}

\subsubsection{Lifts of $\Psi_I$ and $\Theta_I$}
Define
\[
  \sM_I^a \ = \ 
  C_{I,\bC}^{-a-1} \bs \sM_I \,.
\]
This agrees with the definition of $\sM_I^1$ in \eqref{E:MI1}, and $\sM_I^0 = \sD_I$ by \eqref{E:DI}.  Recall that $\sM_I \subset \check \sD$ (\S\ref{S:MI}).  So we may take the intersection $\sM_I \cap \sS$ with the Schubert variety $\sS$ (\S\ref{S:schubert}); we have
\[
  \sM_I \cap \sS \ = \ 
  \exp(\ff^\perp \cap \fc_{I,\bC}) 
  \cdot F \,.
\]
The ``Schubert quotient''
\begin{equation}\label{E:SIa}
  \sS_I^a \ = \ \exp(\ff^\perp \cap \fc_{I,\bC}^{-a-1}) \bs (\sM_I \cap \sS)
  \ \subset \ \sM_I^a
\end{equation}
is Zariski open in $\sM_I^a$.  Lemma \ref{L:Phi-2} yields

\begin{corollary}\label{C:Psi-2}
The restriction of $\Psi_I$ to $\Zc_I \cap X$ lifts to $(\Gamma_{I,\infty}^{-2} \exp(\bC\s_I)) \bs (\sM_I \cap \sS)$.  Likewise, the restriction of $\Theta_I$ to $\Zw_I \cap X$ lifts to $\Gamma_{I,\infty}^{-2} \bs \sS_I^1 = \sS_I^1$.   That is, there is a commutative diagram
\begin{equation} \nonumber 
\begin{tikzcd}[row sep=small]
  & (\Gamma_{I,\infty}^{-2} \exp(\bC\s_I))\bs (\sM_I \cap \sS) 
    \arrow[dr] \arrow[dd]
  & \\
  \Zc_I \cap X \arrow[dd,hook]
  \arrow[ru,"\Psi_{I,\infty}^{-2}"]
  \arrow[rr,"\Psi_I\hspace{30pt}",crossing over]
  & & 
  (\Gamma_I \exp(\bC\s_I))\bs \sM_I 
  \arrow[dd,"\pi_1"]
  \\
  & \sS_I^1 
    \arrow[dr]
  & \\
  \Zw_I \cap X \arrow[rr,"\Theta_I"] 
  \arrow[ru,"\Theta_{I,\infty}^{-2}"]
  & & \Gamma_I \bs \sM_I^1 \,.
\end{tikzcd}
\end{equation}
\end{corollary}

\subsection{The nilpotent logarithms of monodromy along $A'$} \label{S:sX}

Shrinking $X$ if necessary, we may assume that $\Zw_I \cap X$ is closed in $X$.  Recall that $\Zc_J \cap \Zw_I \not= \emptyset$ if and only if $\Zc_J \subset \Zw_I$ (Lemma \ref{L:Wup}).  Define
\[
  \s_{A'} \ = \ 
  \bigcup_{\Zc_J \cap \Zw_I \cap A' \not= \emptyset} \s_J \,.
\]
Note that 
\[
  \s_I \ \subset \ \s_{A'} \ \subset \ 
  \fc_{I,\bQ} \,.
\]
While $\s_I$ is well-defined along $\Zw_I \cap X$ (because the monodromy $\Gamma_{I,\infty}^{-2}$ about $\Zw_I \cap X$ takes value in the centralizer $\Gamma_I$ of the cone $\s_I$), in general the larger cone $\s_{A'}$ will not be well-defined: the cone $\s_{A'}$ is defined only up to the action of $\Gamma_{I,\infty}^{-2}$, and $\Gamma_{I,\infty}^{-2}$ need not centralize all the $\s_J$.  

\begin{lemma} \label{L:sX-2}
We have $\s_{A'} \subset \fc_{I,\bQ}^{-2}$, and $\s_{A'}$ is well-defined modulo $\fc_{I,\bQ}^{-4}$.
\end{lemma}

\begin{proof}
By definition $\Zc_J \subset \Zw_I$ if and only if $J \subset I$ and the weight filtrations coincide $W(\s_I) = W(\s_J)$, cf.~\S\ref{S:Zw}.  This implies $\s_J \subset \fc_{I,\bQ}^{-2}$, cf.~\eqref{E:W(N)}.  Thus $\s_{A'} \subset \fc_{I,\bQ}^{-2}$.  

Because $\Gamma_{I,\infty}^{-2} \subset C_{I,\bQ}^{-2}$, we see from \eqref{E:cIab} that the $\s_{A'}$ is well-defined modulo $\fc_{I,\bQ}^{-4}$.
\end{proof}

\begin{corollary} \label{C:sX-2}
The subgroup $\exp(\bQ\s_{A'} + \fc_{I,\bQ}^{-4}) = \exp(\bQ\s_{A'}) \cdot C_{I,\bQ}^{-4} \subset C_{I,\bQ}^{-2}$ is well defined.
\end{corollary}

\begin{remark} \label{R:sX-2}
We may slightly strengthen Lemma \ref{L:sX-2} as follows.  Shrinking $X$ if necessary, we may assume that $\Zc_J \subset \Zw_I$ intersects $X$ if and only if $\Zc_J$ intersects $A'$.  Since $(W,F)$ is a mixed Hodge structure arising along $A' \cap \Zc_I$, and $\s_J$ polarizes some mixed Hodge structure $(W,F')$ arising along $A' \cap \Zc_J$, we see that $\s_J \subset \fc_{J,F'}^{-1,-1} \subset  \fc_{I,F}^{-1,-1} + \fc_{I,\bQ}^{-4}$.  Thus 
\[
\s_{A'} \subset \fc_{I,F}^{-1,-1} + \fc_{I,\bC}^{-4} \,.
\]
\end{remark}

The following theorem says that $\Psi_I$ is almost completely determined by $\Theta_I$.

\begin{theorem}\label{T:thatsall}
The group $\exp(\bQ\s_{A'})$ is well-defined.  Let $\Gamma_{I,\infty}' = \Gamma_{I,\infty} \cap \exp(\bQ\s_{A'})$.  We have a commutative diagram
\[
\begin{tikzcd}[row sep=small]
  & (\Gamma_{I,\infty}' \exp(\bC\s_I))\bs (\sM_I \cap \sS) 
    \arrow[dr] \arrow[dd]
  & \\
  \Zc_I \cap X \arrow[d,hook]
  \arrow[ru,"\Psi_{I,\infty}'"]
  \arrow[rr,"\Psi_I\hspace{30pt}",crossing over]
  & & 
  (\Gamma_I \exp(\bC\s_I))\bs \sM_I 
  \arrow[ddd,"\pi_1"]
  \\
  \Zw_I \cap X
  \arrow[r,"\Psi_{I,\infty}''"]
  \arrow[dd,equal]
  & (\Gamma_{I,\infty}' \exp(\bC\s_{A'}))\bs (\sM_I \cap \sS) 
  \arrow[d]
  & \\
  & \sS_I^1 
    \arrow[dr]
  & \\
  \Zw_I \cap X \arrow[rr,"\Theta_I"] 
  \arrow[ru,"\Theta_{I,\infty}'"]
  & & \Gamma_I \bs \sM_I^1 \,.
\end{tikzcd}
\]
The map $\Psi_{I,\infty}''$ is locally contant on the fibres of $\Theta_{I,\infty}'$.
\end{theorem}

\begin{remark}\label{R:thatsall}
The information in $\Psi_{I,\infty}'$ that is lost by taking the larger quotient by $\exp(\bC\s_{A'})$ can be recovered as follows.  We will see in the proof of Theorem \ref{T:thatsall} (cf.~\S\ref{S:prf-ta-finish}) that the restriction of the period map $\Phi$ to $U = B \cap X$ lifts to $\Gamma_{I,\infty}'\bs (\sD \cap \sS)$: there is a commutative diagram
\[
  \begin{tikzcd}
  & \Gamma_{I,\infty}' \bs (\sD \cap \sS) \arrow[d] \\
  U \arrow[r,"\Phi"'] \arrow[ru,"\Phi_{I,\infty}'"]
  & \ \Gamma \bs \sD \,.
  \end{tikzcd}
\]    
Fix a limiting mixed Hodge structure $(W,F',\s_J)$ along $A'$.  Complete $N' \in \s_J$ to an $\fsl_2$--triple $\{M',Y',N'\}$, cf.~\S\ref{S:sl2}.  The construction of \S\ref{S:prf-mark2} may be adapted to define line bundles $(\Phi_{I,\infty}')^*(\tilde\Lambda_{M'})$ over $X$, and the sections $(\rho,\tPhi_\infty')^*f_{M'}$ of these lines bundles encode the information in $\Psi_{I,\infty}'$ that is lost by taking the larger quotient by $\exp(\bC\s_{A'})$.
\end{remark}

We have
\[
  \Zc_I \ \subset \ \Zw_I
\]
(Lemma \ref{L:Wup}).  In general, equality need not hold.  When it does, we have $\bC\s_{A'} = \bC\s_I$.

\begin{corollary} \label{C:thatsall}
Suppose that $\Zc_I = \Zw_I$.  The map $\Psi_I : \Zc_I \to (\Gamma_I \exp(\bC\s_I)) \bs \sM_I$ is locally constant on the fibres of $\Theta_I : \Zc_I \to \Gamma_I\bs \sM_I^1$.
\end{corollary}

\begin{remark}
Theorem \ref{T:thatsall} is proved in \S\S\ref{S:prf-ta-start}--\ref{S:prf-ta-finish}.  The basic idea is that we need to show that the functions $\hat g_o$ of \S\ref{S:coord-rep} satisfy $\td \log \hat g_o \equiv 0$ modulo $\bC\s_{A'}$.  Thanks to the infinitesimal period relation \eqref{E:ipr-tg} and Remark \ref{R:sX-2} it suffices to show that $(\td \log \hat g_o)^{-1,\bullet} \equiv 0$ modulo $\bC\s_{A'}$.   The execution is more involved because, a priori, $\exp(\bQ\s_{A'})$ is well-defined only modulo $C_{I,\bQ}^{-4}$ (Lemma \ref{L:sX-2}).  This forces us to work inductively.  We start by showing that the lift $\Psi_{I,\infty}^{-2}$ of $\Psi_I$ in Corollary \ref{C:Psi-2} admits an extension to $\Zw_I \cap X$ modulo the well-defined $\exp(\bC\s_{A'})\cdot C_{I,\bC}^{-3}$.  We use this extension to construct a holomorphic map $\psi : A' \to (\bC^*)^n$.  Since $A'$ is compact, these functions must be constant.  With the infinitesimal period relation this is enough to conclude that the extension is constant along $A'$ (\S\ref{S:prf-ta-start} and Theorem \ref{T:ta-bc}).  This is the base case of the induction, and it is allows us to deduce that $\exp(\bQ\s_{A'})$ is well-defined modulo the smaller group $C_{I,\bQ}^{-5}$ (Corollary \ref{C:sX-3}).
\end{remark}

\subsection{Extending a quotient of $\Psi_{I,\infty}^{-2}$ to $\Zw_I \cap X$}\label{S:prf-ta-start}

Given $\Zc_J \subset \Zw_I$, it follows from Lemma \ref{L:Phi-2} and Remark \ref{R:FA'} that the restriction of $\Psi_J : \Zc_J \to (\Gamma_J \exp(\bC\s_J)) \bs M_J$ to $\Zc_J \cap X$ lifts to $((\Gamma_{I,\infty}^{-2} \cap\Gamma_J) \exp(\bC\s_J)) \bs (M_J \cap \sS)$; we have a commutative diagram
\[
\begin{tikzcd}
  & ((\Gamma_{I,\infty}^{-2} \cap\Gamma_J) \exp(\bC\s_J)) \bs (M_J \cap \sS)
  \arrow[d]\\
  \Zc_J \cap X \arrow[r,"\Psi_J"']
  \arrow[ur,end anchor={south west}]
  & (\Gamma_J \exp(\bC\s_J)) \bs M_J \,.
\end{tikzcd}
\]
It follows from Lemma \ref{L:MJinMI} and Corollary \ref{C:sX-2} that we have a well-defined map
\[
  ((\Gamma_{I,\infty}^{-2} \cap\Gamma_J) \exp(\bC\s_J)) \bs (M_J \cap \sS)
  \ \to \ 
  (\Gamma_{I,\infty}^{-2} \exp(\bC\s_{A'})) \bs \sS_I^2 \,.
\]
Composing the lift with this map defines
\[
  \Psi_{J,\infty}^{-2} : 
  \Zc_J \cap X \ \to \ 
  (\Gamma_{I,\infty}^{-2} \exp(\bC\s_{A'})) \bs 
  \sS_I^2 \,.
\]
Define a holomorphic map
\[
  \Psi_{X,\infty}^{-2} : \Zw_I \cap X 
  \ \to \ (\Gamma_{I,\infty}^{-2} \exp(\bC\s_{A'})) \bs \sS_I^2 
\]
by specifying $\left.\Psi_{X,\infty}^{-2}\right|_{\Zc_J} = \Psi_{J,\infty}^{-2}$.  This map is the desired extension of (a quotient of) $\Psi_{I,\infty}^{-2}$:  we have a commutative diagram
\begin{equation}\label{E:PsiX-2}
\begin{tikzcd}
     \Zc_I \cap X
    \arrow[d,hook] \arrow[r,"\Psi_{I,\infty}^{-2}"]
    & (\Gamma_{I,\infty}^{-2} \exp(\bC\s_I)) \bs (\sM_I \cap \sS) \arrow[d]
    \\
    \Zw_I \cap X
    \arrow[r,"\Psi_{X,\infty}^{-2}"]
    \arrow[dr,"\Theta_{I,\infty}^{-2}"']
    & (\Gamma_{I,\infty}^{-2} \exp(\bC\s_{A'})) \bs \sS_I^2 \arrow[d,"\hat\pi_2"]
    \\
    & \sS_I^1\,.
\end{tikzcd}
\end{equation}

\begin{theorem} \label{T:ta-bc}
The map $\Psi_{X,\infty}^{-2}$ of \eqref{E:PsiX-2} is constant along the $\Theta_I$--fibre $A' \subset \Zw_I$.   
\end{theorem}

\noindent Before proving the theorem (in \S\ref{S:prf-tabc}), we discuss how Theorem \ref{T:ta-bc} allows us to bootstrap to the next step.  

\subsection{Bootstrapping from Theorem \ref{T:ta-bc}} \label{S:bs2-3}

By Corollary \ref{C:sX-2}, the subgroup 
\begin{equation}\label{E:sA'}
  \exp(\bQ\s_{A'} + \fc_{I,\bQ}^{-3}) 
  \ = \ \exp(\bQ\s_{A'})\cdot C_{I,\bQ}^{-3} 
  \ \subset \ C_{I,\bQ}^{-2}
\end{equation}
is well defined.  Let
\begin{equation}\label{E:GammaI-3}
  \Gamma_{I,\infty}^{-3} \ = \ 
  \Gamma_{I,\infty} \,\cap\, 
  \left(
  \exp(\bQ\s_{A'})\cdot C_{I,\bQ}^{-3}
  \right) \ \subset \ 
  \Gamma_{I,\infty}^{-2}\,.
\end{equation}
As a corollary of Theorem \ref{T:ta-bc} we obtain the following strengthening of Lemma \ref{L:Phi-2}.

\begin{corollary} \label{C:Phi-3}
There exists a neighborhood $X \subset \olB$ so that the restriction of the period map $\Phi$ to $U = B \cap X$ lifts to $\Gamma_{I,\infty}^{-3}\bs (\sD \cap \sS)$: there is a commutative diagram
\[
  \begin{tikzcd}
  & \Gamma_{I,\infty}^{-3} \bs (\sD \cap \sS) \arrow[d] \\
  U \arrow[r,"\Phi"'] \arrow[ru,"\Phi_{I,\infty}^{-3}"]
  & \ \Gamma \bs \sD \,.
  \end{tikzcd}
\]
\end{corollary}

\begin{proof}
The proof of Lemma \ref{L:Phi-2} applies here, with the further refinement (made possible by Theorem \ref{T:ta-bc}) we may choose the local lifts so that $(F^p_o \cap W_\ell)/(F^p_o \cap W_{\ell-3})$ is well-defined modulo $\exp(\bC\s_{A'})$.
\end{proof}

Corollary \ref{C:Phi-3} in turn yields strenthenings of Corollaries \ref{C:Psi-2} and \ref{C:sX-2}:

\begin{corollary}\label{C:Psi-3}
The restriction of $\Psi_I$ to $\Zc_I \cap X$ lifts to $(\Gamma_{I,\infty}^{-3} \exp(\bC\s_I)) \bs (\sM_I \cap \sS)$.  Likewise, the restriction of $\Theta_I$ to $\Zw_I \cap X$ lifts to $\Gamma_{I,\infty}^{-3} \bs \sS_I^1 = \sS_I^1$.   That is, there is a commutative diagram
\begin{equation} \nonumber
\begin{tikzcd}[row sep=small]
  & (\Gamma_{I,\infty}^{-3} \exp(\bC\s_I))\bs (\sM_I \cap \sS) 
    \arrow[dr] \arrow[dd]
  & \\
  \Zc_I \cap X \arrow[dd,hook]
  \arrow[ru,"\Psi_{I,\infty}^{-3}"]
  \arrow[rr,"\Psi_I\hspace{30pt}",crossing over]
  & & 
  (\Gamma_I \exp(\bC\s_I))\bs \sM_I 
  \arrow[dd,"\pi_1"]
  \\
  & \sS_I^1 
    \arrow[dr]
  & \\
  \Zw_I \cap X \arrow[rr,"\Theta_I"] 
  \arrow[ru,"\Theta_{I,\infty}^{-3}"]
  & & \Gamma_I \bs \sM_I^1 \,.
\end{tikzcd}
\end{equation}
\end{corollary}

\begin{corollary}\label{C:sX-3}
The subgroup $\exp(\bQ\s_{A'} + \fc_{I,\bQ}^{-5}) = \exp(\bQ\s_{A'}) \cdot C_{I,\bQ}^{-5} \subset C_{I,\bQ}^{-2}$ is well-defined.
\end{corollary}

\begin{proof}
Corollary \ref{C:Phi-3} implies that $\s_{A'}$ is well-defined modulo $\Gamma_{I,\infty}^{-3}$.  Suppose that $\gamma \in \Gamma_{I,\infty}^{-3}$.  Then \eqref{E:cIab} and \eqref{E:GammaI-3} imply $\exp( \bC\tAd_\gamma \s_{A'})$ is equivalent to $\exp(\bC\s_{A'})$ modulo $C_{I,\bC}^{-5}$.  
\end{proof}

As in \S\ref{S:prf-ta-start}, we may extend a quotient of $\Psi_{I,\infty}^{-3}$ to $\Zw_I \cap X$.  The new extension \eqref{E:PsiX-3} is an improvement over the previous extension \eqref{E:PsiX-2} because we obtain the extension after quotienting by a \emph{smaller} group ($C_I^{-4} \subset C_I^{-3}$).  Given $\Zc_J \subset \Zw_I$, it follows from Corollary \ref{C:Phi-3} and Remark \ref{R:FA'} that the restriction of $\Psi_J : \Zc_J \to (\Gamma_J \exp(\bC\s_J)) \bs M_J$ to $\Zc_J \cap X$ lifts to $((\Gamma_{I,\infty}^{-3} \cap\Gamma_J) \exp(\bC\s_J)) \bs (M_J \cap \sS)$; we have a commutative diagram
\[
\begin{tikzcd}
  & ((\Gamma_{I,\infty}^{-3} \cap\Gamma_J) \exp(\bC\s_J)) \bs (M_J \cap \sS)
  \arrow[d]\\
  \Zc_J \cap X \arrow[r,"\Psi_J"']
  \arrow[ur,end anchor={south west}]
  & (\Gamma_J \exp(\bC\s_J)) \bs M_J \,.
\end{tikzcd}
\]
It follows from Lemma \ref{L:MJinMI} and Corollary \ref{C:sX-3} that we have a well-defined map
\[
  ((\Gamma_{I,\infty}^{-3} \cap\Gamma_J) \exp(\bC\s_J)) \bs (M_J \cap \sS)
  \ \to \ 
  (\Gamma_{I,\infty}^{-3} \exp(\bC\s_{A'})) \bs \sS_I^3 \,.
\]
Composing the lift with this map defines
\[
  \Psi_{J,\infty}^{-3} : 
  \Zc_J \cap X \ \to \ 
  (\Gamma_{I,\infty}^{-3} \exp(\bC\s_{A'})) \bs 
  \sS_I^3 \,.
\]
Define a holomorphic map
\[
  \Psi_{X,\infty}^{-3} : \Zw_I \cap X 
  \ \to \ (\Gamma_{I,\infty}^{-3} \exp(\bC\s_{A'})) \bs \sS_I^3
\]
by specifying $\left.\Psi_{X,\infty}^{-3}\right|_{\Zc_J} = \Psi_{J,\infty}^{-3}$.  We have a commutative diagram:
\begin{equation} \label{E:PsiX-3}
\begin{tikzcd}
     \Zc_I \cap X
    \arrow[d,hook] \arrow[r,"\Psi_{I,\infty}^{-3}"]
    & (\Gamma_{I,\infty}^{-3} \exp(\bC\s_I)) \bs (\sM_I \cap \sS) \arrow[d]
    \\
    \Zw_I \cap X
    \arrow[r,"\Psi_{X,\infty}^{-3}"]
    \arrow[ddr,"\Theta_{I,\infty}^{-3}"',end anchor={north west}]
    & (\Gamma_{I,\infty}^{-3} \exp(\bC\s_{A'})) \bs \sS_I^3 \arrow[d,"\hat\pi_3"]
    \\ 
    & \exp(\bC\s_{A'}) \bs \sS_I^2
    \arrow[d]
    \\
    & \sS_I^1\,.
\end{tikzcd}
\end{equation}

\begin{remark}\label{R:ind-c}
Theorem \ref{T:ta-bc} implies that the composition $\hat\pi_3 \circ \Psi_{X,\infty}^{-3}$ is constant on $A'$.  The next inductive step would be to prove that the map $\Psi_{X,\infty}^{-3}$ of \eqref{E:PsiX-3} is constant along $A'$.  See \S\ref{S:ta-ind} for the general case.

\end{remark}  

\subsection{Proof of Theorem \ref{T:ta-bc}} \label{S:prf-tabc}

The basic idea of the proof is: (i) show that the restriction of $\Psi_{X,\infty}^{-2}$ to $A'$ defines a holomorphic map $\psi: A' \to (\bC^*)^n$; (ii) since $A'$ is compact $\psi$ must be constant; and (iii) the infinitesimal period relation implies that the restriction of $\Psi_{X,\infty}^{-2}$ to $A'$ is constant.  We now proceed with the details.

The biholomorphism $\lambda : \sS \to \ff^\perp$ of \S\ref{S:pmr} restricts to a biholomorphism 
\[
  \lambda : \sM_I \cap \sS 
  \ \to \ \ff^\perp \cap \fc_{I,\bC} \,,
\]
and induces a biholomophism
\begin{equation}\label{E:lambda}
  \lambda_k : \sS_I^k
  \ \to \ 
  \frac{\ff^\perp \cap \fc_{I,\bC}}
  {\ff^\perp \cap \fc_{I,\bC}^{-k-1}}
  \ \simeq \ 
  \bigoplus_{\substack{p<0\\-k \le p+q \le 0}} \fc_{I,F}^{p,q} \,.
\end{equation}
From \S\ref{S:coord-rep}, and the definition of $\Psi_{X,\infty}^{-2}$ in \S\ref{S:prf-ta-start}, we see that the local coordinate representation of $\Psi_{X,\infty}^{-2}$ is 
\[
  \Psi_{X,\infty}^{-2}(0,w) \ \equiv \ 
  \log \hat g(0,w)
  \ (\hbox{mod } \bC \s_{A'} \op (\ff^\perp \cap \fc_{I,\bC}^{-3}))
  \ \in \ \frac{\ff^\perp \cap \fc_{I,\bC}}{\bC\s_{A'} \,\op\, (\ff^\perp \cap \fc_{I,\bC}^{-3})} \,.
\]
We see from \eqref{E:lambda} that in order to prove the theorem (that $\Psi_{X,\infty}^{-2}$ is contant along $A'$) it suffices to show that 
\begin{equation}\label{E:cnst-2}
  \left.\left(\td \log \hat g(0,w)\right)^{p,q}
  \right|_{A' \cap \olU_o} 
  \ \equiv \ 0 
  \quad\hbox{mod}\quad \bC\s_{A'} 
  \,,\quad \hbox{for all} \quad p+q \ge -2\,.
\end{equation}

Let 
\begin{equation}\label{E:tpq}
  \tilde p_k : \sM_I \cap \sS \to \sS_I^k
  \tand
  \tilde q_{k+1} : \sS_I^{k+1} \to \sS_I^k
\end{equation}
be the natural projections, cf.~\eqref{E:SIa}.  The biholomorphisms $\lambda_1$ and $\lambda_2$ of \eqref{E:lambda} identify the fibre of $\tilde q_2 : \sS_I^2 \to \sS_I^1$ over $\tilde p_1(F) \in \sS_I^1$ with
\[
  \tilde q{}_2^{-1}(\tilde p_1(F)) \ \simeq \ 
  \frac{\ff^\perp \cap \fc_{I,\bC}^{-2}}{\ff^\perp \cap \fc_{I,\bC}^{-3}}
  \ \simeq \   \bigoplus_{\substack{p<0\\p+q=-2}}
  \fc_{I,F}^{p,q} \,.
\]
The second identification follows from \eqref{E:ds-cIa} and \eqref{E:ffperp}.  The fibre of $\hat q_2 : \exp(\bC\s_{A'}) \bs \sS_I^2 \to \sS_I^1$ over $\tilde p_1(F)$ is 
\[
  \hat q{}_2^{-1}(\tilde p_1(F)) \ \simeq \ 
  \frac{\ff^\perp \cap \fc_{I,\bC}^{-2}}{\bC \s_{A'} \,\oplus\, (\ff^\perp \cap \fc_{I,\bC}^{-3}) } \,.
\]
The group $\Gamma_{I,\infty}^{-2}$ naturally acts on this fibre.  Then the fibre of $\hat \pi_2 : (\Gamma_{I,\infty}^{-2}\exp(\bC\s_{A'})) \bs \sS_I^2 \to \sS_I^1$ over $\tilde p_1(F) \in \sS_I^1$ is the quotient of $\hat q{}_2^{-1}(\tilde p_1(F))$ by the action of $\Gamma_{I,\infty}^{-2}$.  

In fact, the larger group $\Gamma_I^{-2} = \Gamma_I \cap C_{I,\bQ}^{-2}$ acts on the fibre $\hat q{}_2^{-1}(\tilde p_1(F))$.  This factors through an action of $\Gamma_I^{-3} \bs \Gamma_I^{-2} \simeq \fc_{I,\bZ}^{-3} \bs \fc_{I,\bZ}^{-2}$.  (The identification is induced by the logarithm $\log : \Gamma_I^{-a} \to \fc_{I,\bQ}^{-a}$.)  So we have a natural projection
\begin{equation}\label{E:proj1}
  \hat \pi{}_2^{-1}(\tilde p_1(F)) \ \to \ 
  \frac{\ff^\perp \cap \fc_{I,\bC}^{-2}}{\bC \s_{A'}\,+\, (\ff^\perp \cap \fc_{I,\bC}^{-3}) \,+\, \fc_{I,\bZ}^{-2}} \,.
\end{equation}
The right-hand side is isomorphic to a complex torus $\bT \times (\bC^*)^n$ with compact factor $\bT$ and noncompact factor $(\bC^*)^n$.   It follows from \eqref{E:ds-conj} that 
\[
  \mathrm{image}\left\{ 
  \fc_{I,F}^{-1,-1} \ \inj \ 
  \ff^\perp \cap \fc_{I,\bC}^{-2} \ \to \
  \frac{\ff^\perp \cap \fc_{I,\bC}^{-2}}{\bC \s_{A'}\,+\, (\ff^\perp \cap \fc_{I,\bC}^{-3}) \,+\, \fc_{I,\bZ}^{-2}}
  \right\} \ \simeq \ (\bC^*)^n \,.
\]
For our purposes it is convenient to work with the equivalent observation that we have a projection
\begin{equation}\label{E:proj2}
  \frac{\ff^\perp \cap \fc_{I,\bC}^{-2}}{\bC \s_{A'}\,+\, (\ff^\perp \cap \fc_{I,\bC}^{-3}) \,+\, \fc_{I,\bZ}^{-2}}
  \ \to \ 
  \frac{\ff^\perp \cap \fc_{I,\bC}^{-2}}{\bC \s_{A'}\,+\, (\ff^\perp \cap \fc_{I,\bC}^{-3}) \,+\, 
  (\fc_{I,F}^{-1,-1})^\perp \,+\,\fc_{I,\bZ}^{-2}}
  \ \simeq \ (\bC^*)^n \,,
\end{equation}
where
\[
  (\fc_{I,F}^{-1,-1})^\perp \ = \ \bigoplus_{\substack{p\le -2\\p+q=-2}}
  \fc_{I,F}^{p,q}
  \ = \ 
  \fc_{I,F}^{-2,0} \,\op\,
  \fc_{I,F}^{-3,1} \,\op\,
  \fc_{I,F}^{-4,2} \,\op \cdots
\]

The restriction of $\Psi_{X,\infty}^{-2}$ to $A'$ takes value in the fibre $\hat \pi{}_2^{-1}(\tilde p_1(F))$.  Composing with the projections \eqref{E:proj1} and \eqref{E:proj2}, we obtain an analytic map
$\psi : A' \to (\bC^*)^n$.  Since $A'$ is compact and connected, the map must be constant.  Locally this map is given by 
\[
  (0,w) \ \mapsto \ [\log \hat g(0,w)]
  \ \in \ 
  \frac{\ff^\perp \cap \fc_{I,\bC}^{-2}}{\bC \s_{A'}\,+\, (\ff^\perp \cap \fc_{I,\bC}^{-3}) \,+\, 
  (\fc_{I,F}^{-1,-1})^\perp \,+\,\fc_{I,\bZ}^{-2}}
  \ \simeq \ (\bC^*)^n \,,
\]
for all $(0,w) \in A' \cap \olU_o$.  Since this map is constant, we necessarily have
\[
  \left.\left(\td \log \hat g(0,w)\right)^{-1,-1}
  \right|_{A' \cap \olU_o} 
  \ \equiv \ 0 
  \quad\hbox{mod}\quad \bC\s_{A'}\,.
\]
Then the infinitesimal period relation (specifically Remark \ref{R:ipr-tg*} with $q=1$, and \S\ref{S:coord-rep}.\ref{i:dg-2}) establishes the desired \eqref{E:cnst-2}. \hfill\qed

\subsection{The inductive hypothesis} \label{S:indhyp}

Fix $k \ge 3$.  We have three inductive hypotheses:
\begin{a_list}
\item \label{i:hyp1}
Assume that the subgroup $\exp(\bQ\s_{A'} + \fc_{I,\bQ}^{-k-1}) = \exp(\bQ\s_{A'})\cdot C_{I,\bQ}^{-k-1} \subset C_{I,\bQ}^{-2}$ is well-defined.  
\end{a_list}

\noindent Define
\begin{equation}\label{E:GammaI-a}
  \Gamma_{I,\infty}^{-k} \ = \ 
  \Gamma_{I,\infty} \,\cap\, 
  \left(
  \exp(\bQ\s_{A'})\cdot C_{I,\bQ}^{-k}
  \right) \ \subset \ \Gamma_{I,\infty}^{1-k}\,.
\end{equation}

\begin{a_list}
\setcounter{enumi}{1}
\item \label{i:hyp2}
Assume that there exists a neighborhood $X \subset \olB$ of $A'$ so that the restriction of the period map $\Phi$ to $U = B \cap X$ lifts to $\Gamma_{I,\infty}^{-k}\bs (\sD \cap \sS)$: there is a commutative diagram
\[
  \begin{tikzcd}
  & \Gamma_{I,\infty}^{-k} \bs (\sD \cap \sS) \arrow[d] \\
  U \arrow[r,"\Phi"'] \arrow[ru,"\Phi_{I,\infty}^{-k}"]
  & \ \Gamma \bs \sD \,.
  \end{tikzcd}
\]
\end{a_list}

\noindent From hypothesis \ref{i:hyp2} we obtain Lemma \ref{L:Psi-k}, the general version of Corollaries \ref{C:Psi-2} and \ref{C:Psi-3}.

\begin{lemma}\label{L:Psi-k}
The restriction of $\Psi_I$ to $\Zc_I \cap X$ lifts to $(\Gamma_{I,\infty}^{-k} \exp(\bC\s_I)) \bs (\sM_I \cap \sS)$.  Likewise, the restriction of $\Theta_I$ to $\Zw_I \cap X$ lifts to $\Gamma_{I,\infty}^{-k} \bs \sS_I^1 = \sS_I^1$.   That is, there is a commutative diagram
\begin{equation} \nonumber
\begin{tikzcd}[row sep=small]
  & (\Gamma_{I,\infty}^{-k} \exp(\bC\s_I))\bs (\sM_I \cap \sS) 
    \arrow[dr] \arrow[dd]
  & \\
  \Zc_I \cap X \arrow[dd,hook]
  \arrow[ru,"\Psi_{I,\infty}^{-k}"]
  \arrow[rr,"\Psi_I\hspace{30pt}",crossing over]
  & & 
  (\Gamma_I \exp(\bC\s_I))\bs \sM_I 
  \arrow[dd,"\pi_1"]
  \\
  & \sS_I^1 \arrow[dr]
  & \\
  \Zw_I \cap X \arrow[rr,"\Theta_I"] 
  \arrow[ru,"\Theta_{I,\infty}^{-k}"]
  & & \Gamma_I \bs \sM_I^1 \,.
\end{tikzcd}
\end{equation}
\end{lemma}

\noindent Hypothesis \ref{i:hyp2} also yields the general version of Corollaries \ref{C:sX-2} and \ref{C:sX-3}:

\begin{lemma}\label{L:sX-a}
The subgroup $\exp(\bC\s_{A'} + \fc_{I,\bQ}^{-k-2}) = \exp(\bQ\s_{A'}) \cdot C_{I,\bQ}^{-k-2} \subset C_{I,\bQ}^{-2}$ is well-defined.
\end{lemma}

\begin{proof}
The inductive hypothesis \ref{i:hyp2} implies that $\s_{A'}$ is well-defined modulo $\Gamma_{I,\infty}^{-k}$.  Suppose that $\gamma \in \Gamma_{I,\infty}^{-k}$.  Then \eqref{E:cIab} and \eqref{E:GammaI-a} imply $\exp( \bC\tAd_\gamma \s_{A'})$ is equivalent to $\exp(\bC\s_{A'})$ modulo $C_{I,\bC}^{-k-2}$.  
\end{proof}

We now have everything we need to extend a quotient of the map $\Psi_{I,\infty}^{-k}$ in Lemma \ref{L:Psi-k} to $\Zw_I \cap X$.  Given $\Zc_J \subset \Zw_I$, it follows from the inductive hypothesis \ref{i:hyp2} and Remark \ref{R:FA'} that the restriction of $\Psi_J : \Zc_J \to (\Gamma_J \exp(\bC\s_J)) \bs M_J$ to $\Zc_J \cap X$ lifts to $((\Gamma_{I,\infty}^{-k} \cap\Gamma_J) \exp(\bC\s_J)) \bs (M_J \cap \sS)$; we have a commutative diagram
\[
\begin{tikzcd}
  & ((\Gamma_{I,\infty}^{-k} \cap\Gamma_J) \exp(\bC\s_J)) \bs (M_J \cap \sS)
  \arrow[d]\\
  \Zc_J \cap X \arrow[r,"\Psi_J"']
  \arrow[ur,end anchor={south west}]
  & (\Gamma_J \exp(\bC\s_J)) \bs M_J \,.
\end{tikzcd}
\]
It follows from Lemmas \ref{L:MJinMI} and \ref{L:sX-a} that we have a well-defined map
\[
  ((\Gamma_{I,\infty}^{-k} \cap\Gamma_J) \exp(\bC\s_J)) \bs (M_J \cap \sS)
  \ \to \ 
  (\Gamma_{I,\infty}^{-k} \exp(\bC\s_{A'})) \bs \sS_I^k \,.
\]
Composing the lift with this map defines
\[
  \Psi_{J,\infty}^{-k} : 
  \Zc_J \cap X \ \to \ 
  (\Gamma_{I,\infty}^{-k} \exp(\bC\s_{A'})) \bs 
  \sS_I^k \,.
\]
Define a holomorphic map
\[
  \Psi_{X,\infty}^{-k} : \Zw_I \cap X 
  \ \to \ (\Gamma_{I,\infty}^{-k} \exp(\bC\s_{A'})) \bs \sS_I^k
\]
by specifying $\left.\Psi_{X,\infty}^{-k}\right|_{\Zc_J} = \Psi_{J,\infty}^{-k}$.  We have a commutative diagram:
\begin{equation} \label{E:PsiX-a}
\begin{tikzcd}
     \Zc_I \cap X
    \arrow[d,hook] \arrow[r,"\Psi_{I,\infty}^{-k}"]
    & (\Gamma_{I,\infty}^{-k} \exp(\bC\s_I)) \bs (\sM_I \cap \sS) \arrow[d]
    \\
    \Zw_I \cap X
    \arrow[r,"\Psi_{X,\infty}^{-k}"]
    \arrow[ddr,"\Theta_{I,\infty}^{-k}"',end anchor={west},bend right=10]
    & (\Gamma_{I,\infty}^{-k} \exp(\bC\s_{A'})) \bs \sS_I^k
    \arrow[d,"\hat\pi_k"]
    \\ 
    & \exp(\bC\s_{A'}) \bs \sS_I^{k-1}
    \arrow[d]
    \\
    & \sS_I^1\,.
\end{tikzcd}
\end{equation}

\noindent This brings us to our third, and final, inductive hypothesis:

\begin{a_list}
\setcounter{enumi}{2}
\item \label{i:hyp3}
Assume that the composition $\hat\pi_k \circ \Psi_{X,\infty}^{-k}$ is constant along $A'$.
\end{a_list}

\begin{lemma}
The inductive hypotheses hold for $k=3$.
\end{lemma}

\begin{proof}
For hypothesis \ref{i:hyp1} see Corollary \ref{C:sX-2} and \eqref{E:GammaI-3}.  For hypothesis \ref{i:hyp2} see Corollary \ref{C:Phi-3}.  For hypothesis \ref{i:hyp3} see Remark \ref{R:ind-c}.
\end{proof}

\subsection{The inductive step} \label{S:ta-ind}

In the case $k=3$, the three inductive hypotheses of \S\ref{S:indhyp} are all corollaries of Theorem \ref{T:ta-bc}, cf.~\S\ref{S:bs2-3}.  The constructions of \S\ref{S:bs2-3} can be adapted in a straightforward way to show that, in order to establish the inductive step, it suffices to prove Theorem \ref{T:indstep}.

\begin{theorem} \label{T:indstep}
The map $\Psi_{X,\infty}^{-k}$ of \eqref{E:PsiX-a} is constant along $A'$.
\end{theorem}

\begin{proof}
As in the proof of Theorem \ref{T:ta-bc}, basic idea is: (i) show that the restriction of $\Psi_{X,\infty}^{-k}$ to $A'$ defines a holomorphic map $\psi: A' \to (\bC^*)^n$; (ii) since $A'$ is compact $\psi$ must be constant; and (iii) the infinitesimal period relation implies that the restriction of $\Psi_{X,\infty}^{-k}$ to $A'$ is constant.

From \S\ref{S:coord-rep}, and the definition of $\Psi_{X,\infty}^{-k}$, we see that the local coordinate representation of $\Psi_{X,\infty}^{-k}$ is 
\[
  \Psi_{X,\infty}^{-k}(0,w) \ \equiv \ 
  \log \hat g(0,w)
  \ (\hbox{mod } \bC \s_{A'} \op (\ff^\perp \cap \fc_{I,\bC}^{-k-1}))
  \ \in \ \frac{\ff^\perp \cap \fc_{I,\bC}}{\bC\s_{A'} \,\op\, (\ff^\perp \cap \fc_{I,\bC}^{-k-1})} \,.
\]
We see from \eqref{E:lambda} that in order to prove the theorem (that $\Psi_{X,\infty}^{-k}$ is contant along $A'$) it suffices to show that 
\begin{equation} \nonumber 
  \left.\left(\td \log \hat g(0,w)\right)^{p,q}
  \right|_{A' \cap \olU_o} 
  \ \equiv \ 0 
  \quad\hbox{mod}\quad \bC\s_{A'} 
  \,,\quad \hbox{for all} \quad p+q \ge -k\,.
\end{equation}
The local coordinate representation of $\hat\pi_k \circ \Psi_{X,\infty}^{-k}$ is 
\[
  \hat\pi_k \circ \Psi_{X,\infty}^{-k}(0,w) \ \equiv \ 
  \log \hat g(0,w)
  \ (\hbox{mod } \bC \s_{A'} \op (\ff^\perp \cap \fc_{I,\bC}^{-k}))
  \ \in \ \frac{\ff^\perp \cap \fc_{I,\bC}}{\bC\s_{A'} \,\op\, (\ff^\perp \cap \fc_{I,\bC}^{-k})} \,.
\]
By inductive hypothesis \S\ref{S:indhyp}\ref{i:hyp3}, we have
\begin{equation}\label{E:cnst-k+1}
  \left.\left(\td \log \hat g(0,w)\right)^{p,q}
  \right|_{A' \cap \olU_o} 
  \ \equiv \ 0 
  \quad\hbox{mod}\quad \bC\s_{A'} 
  \,,\quad \hbox{for all} \quad p+q \ge 1-k\,.
\end{equation}
By definition \eqref{E:sA'} we have $W = W(N)$ for every $N \in \s_{A'}$.  This implies
\begin{equation}\label{E:int-sX}
  \s_{A'} \cap \fc_{I,\bC}^{-3} \ = \ \emptyset \,,
\end{equation}
else \eqref{E:Niso} fails.  So in order to prove Theorem \ref{T:indstep}, it suffices to show that 
\begin{equation}\label{E:indstep}
  \left.\left(\td \log \hat g(0,w)\right)^{p,q}
  \right|_{A' \cap \olU_o} 
  \ = \ 0  
  \,,\quad \hbox{for all} \quad p+q = -k\,.
\end{equation}

Recall the projections $\tilde p_k : \sM_I \cap \sS \to \sS_I^{k-1}$ and $\tilde q_k : \sS_I^k \to \sS_I^{k-1}$ of \eqref{E:tpq}.  The biholomorphisms $\lambda_{k-1}$ and $\lambda_k$ identify the fibre of $\tilde q_k : \sS_I^k \to \sS_I^{k-1}$ over $\tilde p_{k-1}(F) \in \sS_I^{k-1}$ with
\begin{equation}\label{E:tkf}
  \tilde q{}_k^{-1}(\tilde p_{k-1}(F)) \ \simeq \ 
  \frac{\ff^\perp \cap \fc_{I,\bC}^{-k}}{\ff^\perp \cap \fc_{I,\bC}^{-k-1}}
  \ \simeq \ 
  \bigoplus_{\substack{p<0\\p+q=-k}}
  \fc_{I,F}^{p,q} \,.
\end{equation}
The second identification follows from \eqref{E:ds-cIa} and \eqref{E:ffperp}.    Let $\hat p_{k-1} : \sM_I \cap \sS \to \exp(\bC\s_{A'}) \bs \sS_I^{k-1}$ be the natural projection.  From \eqref{E:int-sX} we see that the fibre of $\hat q_k : \exp(\bC\s_{A'}) \bs \sS_I^k \to \exp(\bC\s_{A'}) \bs \sS_I^{k-1}$ over $\hat p_{k-1}(F)\in \exp(\bC\s_{A'}) \bs \sS_I^{k-1}$ is identified with \eqref{E:tkf}.  The group $\Gamma_{I,\infty}^{-k}$ naturally acts on this fibre.  And the fibre of $\hat \pi_k : (\Gamma_{I,\infty}^{-k}\exp(\bC\s_{A'})) \bs \sS_I^k \to \exp(\bC\s_{A'}) \bs \sS_I^{k-1}$ over $\hat p_{k-1}(F) \in \exp(\bC\s_{A'}) \bs \sS_I^{k-1}$ is the quotient of $\tilde q{}_k^{-1}(\tilde p_{k-1}(F))$ by the action of $\Gamma_{I,\infty}^{-k}$.  

In fact, the larger group $\Gamma_I \cap (\exp(\bQ\s_{A'}) \cdot C_{I,\bQ}^{-k})$ acts on the fibre $\tilde q{}_k^{-1}(\tilde p_{k-1}(F))$.  This factors through an action of $(\Gamma_I \cap C_{I,\bQ}^{-k-1}) \bs (\Gamma_I \cap C_{I,\bQ}^{-k}) \simeq \fc_{I,\bZ}^{-k-1} \bs \fc_{I,\bZ}^{-k}$, with the identification induced by the logarithm $\log : \Gamma_I^{-\ell} \to \fc_{I,\bQ}^{-\ell}$.  Keeping \eqref{E:int-sX} in mind, it follows that we have a natural projection
\begin{equation}\label{E:proj1k}
  \hat \pi{}_k^{-1}(\hat p_{k-1}(F)) \ \to \ 
  \frac{\ff^\perp \cap \fc_{I,\bC}^{-k}}{(\ff^\perp \cap \fc_{I,\bC}^{-k-1}) \,+\, \fc_{I,\bZ}^{-k}} \,.
\end{equation}
The right-hand side is isomorphic to a complex torus $\bT \times (\bC^*)^m$ with compact factor $\bT$ and noncompact factor $(\bC^*)^m$.   It follows from \eqref{E:ds-conj} that 
\[
  \mathrm{image}\left\{ 
  \fc_{I,F}^{-1,1-k} \ \inj \ 
  \ff^\perp \cap \fc_{I,\bC}^{-k} \ \to \
  \frac{\ff^\perp \cap \fc_{I,\bC}^{-k}}{(\ff^\perp \cap \fc_{I,\bC}^{-k-1}) \,+\, \fc_{I,\bZ}^{-k}}
  \right\} \ \simeq \ (\bC^*)^n \,,
\]
with $n \le m$.  For our purposes it is convenient to work with the equivalent observation that we have a projection
\begin{equation}\label{E:proj2k}
  \frac{\ff^\perp \cap \fc_{I,\bC}^{-k}}{\ff^\perp \cap \fc_{I,\bC}^{-k-1} \,+\, \fc_{I,\bZ}^{-k}}
  \ \to \ 
  \frac{\ff^\perp \cap \fc_{I,\bC}^{-k}}{(\ff^\perp \cap \fc_{I,\bC}^{-k-1}) \,+\, (\fc_{I,F}^{-1,1-k})^\perp \,+\,\fc_{I,\bZ}^{-k}}
  \ \simeq \ (\bC^*)^n \,,
\end{equation}
where
\[
  (\fc_{I,F}^{-1,1-k})^\perp \ = \ \bigoplus_{\substack{p\le -2\\p+q=-k}}
  \fc_{I,F}^{p,q}
  \ = \ 
  \fc_{I,F}^{-2,2-k} \,\op\,
  \fc_{I,F}^{-3,3-k} \,\op\,
  \fc_{I,F}^{-4,4-k} \,\op \cdots
\]

By inductive hypothesis \S\ref{S:indhyp}\ref{i:hyp3}, the restriction of $\Psi_{X,\infty}^{-k}$ to $A'$ takes value in the fibre $\hat \pi{}_k^{-1}(\hat p_{k-1}(F))$.  Composing with the projections \eqref{E:proj1k} and \eqref{E:proj2k}, we obtain an analytic map
$\psi : A' \to (\bC^*)^n$.  Since $A'$ is compact and connected, the map must be constant.  Locally this map is given by 
\[
  (0,w) \ \mapsto \ [\log \hat g(0,w)]
  \ \in \ 
  \frac{\ff^\perp \cap \fc_{I,\bC}^{-k}}{(\ff^\perp \cap \fc_{I,\bC}^{-k-1}) \,+\, (\fc_{I,F}^{-1,1-k})^\perp \,+\,\fc_{I,\bZ}^{-k}}
  \ \simeq \ (\bC^*)^n \,,
\]
for all $(0,w) \in A' \cap \olU_o$.  Since this map is constant, we necessarily have
\[
  \left.\left(\td \log \hat g(0,w)\right)^{-1,1-k}
  \right|_{A' \cap \olU_o} 
  \ = \ 0 \,.
\]
Then the infinitesimal period relation (specifically Remark \ref{R:ipr-tg*} with $-q=1-k$, and \eqref{E:cnst-k+1}) establishes the desired \eqref{E:indstep}. 
\end{proof}

\subsection{Completing the proof of Theorem \ref{T:thatsall}} \label{S:prf-ta-finish}

Note that $C_I^{-2\sfn-1}$ is trivial. This implies that the inductive process terminates after finitely many steps.  We have $\sS_I^{2\sfn} = \sM_I \cap \sS$.  We obtain the statement of Theorem \ref{T:thatsall} by setting 
\begin{eqnarray*}
    \Gamma_{I,\infty}' & = & \Gamma_{I,\infty}^{-2\sfn} \,,\\
    \Theta_{I,\infty}' & = & \Theta_{I,\infty}^{-2\sfn} \,,\\
    \Psi_{I,\infty}' & = & \Psi_{I,\infty}^{-2\sfn} \,,\\
    \Psi_{I,\infty}'' & = & \Psi_{X,\infty}^{-2\sfn} \,.
\end{eqnarray*}
The map $\Phi_{I,\infty}'$ of Remark \ref{R:thatsall} is $\Phi_{I,\infty}^{-2\sfn}$.  \hfill\qed

\input{global-properties-part2.bbl}


\end{document}

%% file: lists.tex
\newenvironment{a_list}
  {\begin{enumerate}[label=(\alph*),itemsep=3pt,leftmargin=25pt,listparindent=20pt]}
  {\end{enumerate}}
\newenvironment{a.list}
  {\begin{enumerate}[label=\alph*.,itemsep=3pt,leftmargin=25pt,listparindent=20pt]}
  {\end{enumerate}}

\newenvironment{num.list}
  {
  \begin{enumerate}[itemsep=3pt,leftmargin=25pt,listparindent=20pt,label={\arabic*.}]
  }
  {\end{enumerate}}
\newenvironment{i_list}
  {\begin{enumerate}[label=(\roman*),itemsep=3pt,leftmargin=25pt,listparindent=20pt]}
  {\end{enumerate}}
\newenvironment{i_list_emph}
  {\begin{enumerate}[label=\emph{(\roman*)},itemsep=3pt,leftmargin=25pt,listparindent=20pt]}
  {\end{enumerate}}


%% file: macros.tex
\newcommand{\eref}[1]{\emph{\ref{#1}}}

\newcounter{letcnt1} 

\newcounter{letcnt2} 
\counterwithin{letcnt2}{letcnt1}

\def\a{\alpha}

\def\d{\delta}

\def\s{\sigma}
 
\def\x{\xi}

\def\tAd{\mathrm{Ad}} \def\tad{\mathrm{ad}}

\def\tAut{\mathrm{Aut}}

\def\bC{\mathbb{C}} 

\def\fc{\mathfrak{c}}
\def\tcodim{\mathrm{codim}}

\def\td{\mathrm{d}}

\def\tdim{\mathrm{dim}}

\def\tEnd{\mathrm{End}}
\def\cF{\mathcal{F}}

\def\cG{\mathcal{G}}

\def\tGL{\mathrm{GL}}
\def\tGr{\mathrm{Gr}}
\def\fg{\mathfrak{g}}
\def\fgl{\mathfrak{gl}}

\def\sH{\mathscr{H}}
\def\tHom{\mathrm{Hom}}

\def\bi{\mathbf{i}}
\def\cI{\mathcal{I}}
 \def\tIm{\mathrm{Im}}
 \def\tim{\mathrm{im}}
\def\sJ{\mathscr{J}}

\def\tker{\mathrm{ker}}
\def\cL{\mathcal{L}}

\def\sM{\mathscr{M}}

\def\cN{\mathcal{N}}

\def\fp{\mathfrak{p}}

\def\bQ{\mathbb{Q}} 
\def\sQ{\mathscr{Q}}
\def\bR{\mathbb{R}}

\def\tStab{\mathrm{Stab}}

 \def\tspan{\mathrm{span}}
\def\fsl{\mathfrak{sl}}  
 
\def\bT{\mathbb{T}}
\def\cT{\mathcal{T}}

\def\fu{\mathfrak{u}}
\def\cV{\mathcal{V}}

\def\bZ{\mathbb{Z}}

\def\half{\tfrac{1}{2}}

\def\tand{\quad\hbox{and}\quad}
\def\bs{\backslash}

\def\smallb{{\hbox{\small{$\bullet$}}}}
\def\inj{\hookrightarrow}

\def\op{\oplus}
\def\ot{\otimes}

%% file: thms.tex
\newtheorem{corollary}[equation]{Corollary}
\newtheorem{lemma}[equation]{Lemma}

\newtheorem{theorem}[equation]{Theorem}

\newtheorem*{mainthm*}{Main Theorem}
\newtheorem*{theorem*}{Theorem}
\newtheorem*{corollary*}{Corollary}
\newtheorem*{lemma*}{Lemma}

\theoremstyle{definition}

\newtheorem*{boldQ*}{Question}
\newtheorem*{boldP*}{Problem}

\theoremstyle{definition}

\theoremstyle{remark}
\newtheorem{definition}[equation]{Definition}
\newtheorem*{definition*}{Definition}

\newtheorem*{example*}{Example}
\newtheorem{remark}[equation]{Remark}

\newtheorem*{emphQ*}{Question}

%% file: global-properties-part2.bbl
\def\cprime{$'$} \def\Dbar{\leavevmode\lower.6ex\hbox to 0pt{\hskip-.23ex
  \accent"16\hss}D}